\documentclass[11pt]{article}
\usepackage{amsmath, amsthm, amsfonts, amssymb, color}
\usepackage{mathrsfs}
\usepackage{amsfonts, amsmath}
\usepackage{amsmath,amstext,amsthm,amssymb,amsxtra}
\usepackage{txfonts}
\usepackage{tikz}
\usepackage{yhmath}
\usepackage{graphicx}

\usepackage{latexsym,amssymb,amsmath,amsfonts,graphicx,subfigure,epsfig,psfrag,CJK}
\usepackage{float}
\usepackage{psfrag}

\allowdisplaybreaks[4]

\usepackage[colorlinks, citecolor=blue,hypertexnames=false]{hyperref}
\usepackage{pgf,tikz}

\newtheorem{theo}{Theorem}[section]
\newtheorem{lemm}[theo]{Lemma}

\newtheorem{rema}[theo]{Remark}

\numberwithin{equation}{section}

\allowdisplaybreaks

\def\th2{\frac{\theta}{2}}
\begin{document}
\title{\bf Hopf cyclicity and global dynamics for a predator-prey system of Leslie type with simplified Holling type IV functional response\thanks{Supported by the National Science Foundation of China(No.11571379).}}
\author{Yanfei Dai $^{1}$, \quad Yulin Zhao $^{2}$
\thanks{Corresponding author.  E-mail: mcszyl@mail.sysu.edu.cn.}}

\date{\it\footnotesize $^{1}$ School of Mathematics, Sun Yat-sen University, Guangzhou, 510275, P. R. China\\
$^{2}$School of Mathematics (Zhuhai), Sun Yat-sen University, Zhuhai, 519082, P. R. China}

\maketitle 

\begin{abstract}
This paper is concerned with a predator-prey model of Leslie type with simplified Holling type IV functional response, provided that it has either a unique non-degenerate positive equilibrium or three distinct positive equilibria. The type and stability of each equilibrium, Hopf cyclicity of each weak focus, and the number and distribution of limit cycles in the first quadrant are studied.
It is shown that every equilibrium cannot be a center.
If system has  a  unique positive equilibrium which is a weak focus, then its order is at most $2$ and it has Hopf cyclicity $2$. Moreover, some sufficient conditions for the global stability of the unique  equilibrium are established by applying  Dulac's criterion and constructing the Liapunov function.
If system has three distinct positive equilibria, then one of them is a saddle and the others  are both anti-saddles. For two anti-saddles, we prove that the Hopf cyclicity for  positive equilibrium with smaller abscissa (resp.  bigger abscissa) is  $2$ (resp. $1$). Furthermore, if both  anti-saddle positive equilibria are weak foci, then they are unstable multiple foci with multiplicity one.
Moreover, one limit cycle can bifurcate from each of them simultaneously.
Numerical simulations demonstrate  that there is  a big stable limit cycle enclosing these two small limit cycles.
\end{abstract}
\maketitle
\rule[1.5mm]{\textwidth}{0.1pt}
\section{Introduction}
\setcounter{equation}{0}

The dynamic interaction between predators and their prey is one of the
most fundamental interactions in ecology and mathematical ecology due to its universality and importance, see \cite{Volterra1927,May1972}.
Wollkind et al.  \cite{Wollkind1978,Wollkind1988} modeled the population interaction
between the predacious mite \textit{Metaseiulus occidentalis} Nesbitt and its spider
mite prey \textit{Tetranychus mcdanieli} McGregor on fruit trees in Washington State \cite{Hoyt1967,Hoyt1969} and adapted the following ordinary differential equations based on a predator-prey system proposed by May \cite{May1973}
\begin{eqnarray}\label{dai1.1}
\left\{
\begin{array}{ll}
\dot{x}=r x \left(1-\dfrac{x}{K}\right)-p(x) y, \\[6pt]
\dot{y}=s y \left(1-\dfrac{y}{h x}\right),
\end{array}
\right.
\end{eqnarray}
where $x(t)$ and $y(t)$ are the population densities of the
prey and predator at time $t$, respectively.
They assumed that the prey grows logistically with intrinsic growth rate $r$ and carrying capacity $K$ in the absence of predation, the function response is function $p(x)$ with Holling type, and the predator's numerical response is Leslie form originated by Leslie \cite{Leslie1948}  and
grows logistically with intrinsic growth rate $s$ and carrying capacity proportional to
the population of the prey.
When the functional response $p(x)$ is Holling type I, II, III functions in Holling \cite{Holling1965} respectively, Leslie in \cite{Leslie1960},
Hsu and Huang in \cite{Hsu1995,Hsu1998,Hsu1999} obtained some interesting results such as the existence of global attracting positive equilibrium for some parameter values, and the existence and uniqueness of limit cycles for other parameter values, etc.
Collings in \cite{Collings1997} by numerical simulations showed that the bifurcation and stability behaviors of the model \eqref{dai1.1} with Holling type I, II, III functional response are qualitatively similar, but which is different to the dynamical behaviors of the model \eqref{dai1.1} with Holling type IV at higher levels of prey interference.

For Holling type response function, Sokol and Howell in \cite{Sokol1980}  proposed
a two-parameter simplification of Monod-Haldane functional response \cite{Andrews1968} of the form
\begin{equation*}
p(x)=\frac{m x}{x^2+b},
\end{equation*}
and found that it can fit their experimental data significantly better and is simpler since it involves only two parameters.
This function is called  the simplified Holling type IV functional response \cite{Ruan2001,Li2007} and can also be used to describe the phenomenon of ``inhibition'' in microbial dynamics and ``group defence'' in population dynamics.
Li and Xiao \cite{Li2007} considered the Leslie type predator-prey model \eqref{dai1.1} with simplified Holling type IV functional response
\begin{eqnarray}\label{dai1.2}
\left\{
\begin{array}{ll}
\dot{x}=r x \left(1-\dfrac{x}{K}\right)-\dfrac{m x y}{x^2+b}, \\[6pt]
\dot{y}=s y \left(1-\dfrac{y}{h x}\right),
\end{array}
\right.
\end{eqnarray}
where $r,\ K,\ m,\ b,\ s$ and $h$ are all positive parameters.
For simplicity, they scaled $x$, $y$, $t$ and parameters in \eqref{dai1.2} by letting
\begin{equation}\label{dai1.3}
\begin{aligned}
\bar{x}=\frac{x}{K},\,\, \bar{y}=\frac{my}{rK^2},\,\, \bar{t}=rt,\,\, a=\frac{b}{K^2}, \,\, \delta=\frac{s}{r},\,\, \beta=\frac{sK}{hm}.
\end{aligned}
\end{equation}
Dropping the bars,  system \eqref{dai1.2} is equivalent to
\begin{eqnarray}\label{dai1.4}
\left\{
\begin{array}{ll}
\dot{x}=x \left(1-x\right)-\dfrac{x y}{x^2+a}, \\[6pt]
\dot{y}=y \left(\delta-\dfrac{\beta y}{x}\right).
\end{array}
\right.
\end{eqnarray}
Suppose $\bar{E}(\bar{x}, \bar{y})$ is a positive equilibrium of system \eqref{dai1.4}, then $\bar{x}$ is a positive root of the equation
\begin{equation}\label{dai1.5}
{x}^3-{x}^2+\left(a+\frac{\delta}{\beta}\right) x-a=0
\end{equation}
in the interval $(0, 1)$.

We denote the determinant and trace of the Jacobian matrix of system \eqref{dai1.4} at $\bar{E}$ by $\textrm{Det}\left(J(\bar{E})\right)$ and $\textrm{Tr}\left(J(\bar{E})\right)$, respectively. Then $\bar{E}$ is called an elementary equilibrium  if $\textrm{Det}\left(J(\bar{E})\right)\neq0$, otherwise  it is a degenerate equilibrium.
Specially, $\bar{E}$ is called a hyperbolic saddle  if $\textrm{Det}\left(J(\bar{E})\right)<0$,  and called \emph{center or focus type} if $\textrm{Det}\left(J(\bar{E})\right)>0$ and $\textrm{Tr}\left(J(\bar{E})\right)=0$, respectively.
As usual, the positive equilibrium of system \eqref{dai1.2} is said to have \emph{Hopf cyclicity} $k$ if for any biologically meaningful parameters, at most $k$ limit cycles can bifurcate from this equilibrium by Hopf bifurcation, and there exist some choice of biologically meaningful parameters such that $k$ limit cycles can appear near the equilibrium.

For system \eqref{dai1.4}, it can be inferred from \cite{Hsu1995} that if $\bar{x}$ is a multiple positive root of Eq. \eqref{dai1.5} in the interval $(0, 1)$, then $\bar{E}(\bar{x}, \bar{y})$ must be a degenerate positive equilibrium of system \eqref{dai1.4}.
Let $A=1-3(a+\delta/\beta)$ and $\Delta=(1-27a-3A)^2-4A^3$. It is shown in \cite{Li2007} that the following statements hold:
\begin{itemize}
\item[(a)] Suppose $\Delta>0$, then system \eqref{dai1.4} has a unique positive equilibrium, which is an elementary and anti-saddle equilibrium;
\item[(b)] Suppose both $\Delta=0$ and $A=0$, i.e., $a=1/27$, $\delta/\beta=8/27$, then system \eqref{dai1.4} has a unique positive equilibrium, which is a degenerate equilibrium;
\item[(c)] Suppose $\Delta=0$ and $A>0$, then system \eqref{dai1.4} has two distinct positive equilibria.
\begin{itemize}
\item[(c1)] if  $(1-27a-3A)=-2A^{3/2}$, then $E^{*}(x^{*}, y^{*})$ is a degenerate equilibrium, and  $E^{*}_1(x_1^{*}, y_1^{*})$ is an elementary and anti-saddle equilibrium, where $x^{*}<x_1^{*}$;
\item[(c2)] if  $(1-27a-3A)=2A^{3/2}$, then $E^{*}_2(x_2^{*}, y_2^{*})$ is an elementary and anti-saddle equilibrium, and $E^{*}(x^{*}, y^{*})$  is a degenerate equilibrium, where $x_2^{*}<x^{*}$.
\end{itemize}
\item[(d)] Suppose $\Delta<0$, then system \eqref{dai1.4} has three different positive equilibria: one is a saddle and the other two are anti-saddle equilibria.
\end{itemize}
Li and Xiao \cite{Li2007} and Huang et al. \cite{Huang2016} provided detailed analysis on the above case (c) and case (b), respectively. These two cases correspond to the case when system \eqref{dai1.4} has at least one degenerate positive equilibrium.
More precisely, for the case (c), Li and Xiao \cite{Li2007} showed that system \eqref{dai1.4} can have two non-hyperbolic positive
equilibria (one is a cusp of codimension 2 and the other is a multiple focus of multiplicity one) for some values of parameters. They proved that Bogdanov-Takens bifurcation and subcritical Hopf bifurcation can occur simultaneously in the small neighborhoods of these two equilibria, respectively. For the case (b), Huang et al. \cite{Huang2016} showed that
there exists a unique degenerate positive equilibrium which is a degenerate Bogdanov-Takens singularity (focus case) of codimension $3$. They proved that the model exhibits
degenerate focus type Bogdanov-Takens bifurcation of codimension $3$ around this equilibrium.  These results not only showed that system \eqref{dai1.2} has rich and complicated dynamics, but also supported the numerical conclusion of Collings in \cite{Collings1997} that the dynamical behavior
of system \eqref{dai1.2} with Holling type IV functional response is  different from that of system \eqref{dai1.2} with Holling types I, II and III functional responses.

However, the Hopf bifurcation and global dynamics of system \eqref{dai1.2}
have not been discussed theoretically in Collings \cite{Collings1997}, Li and Xiao \cite{Li2007} and Huang et al. \cite{Huang2016} for the two cases (a) and (d) mentioned above, that is, the cases (I) when system \eqref{dai1.2} has a unique non-degenerate positive equilibrium; and (II) when system \eqref{dai1.2} has three distinct positive equilibria.
In this paper, we  investigate the number and distribution of limit cycles of system \eqref{dai1.2},  and the global dynamics in the first quadrant for these two cases.
By scalings of the coordinates,  system \eqref{dai1.2} can be transformed into an equivalent polynomial differential system in the interior of the first quadrant.
To get the global behavior of the equivalent system, the qualitative behavior near the origin is investigated.
When the equivalent system has a unique non-degenerate positive equilibrium, we show  that two limit cycles can bifurcate from this equilibrium and establish some sufficient conditions for the global stability of this equilibrium.
When the equivalent system has three distinct positive equilibria, one of them  is a saddle and the others  are both anti-saddle,  two  (resp. one )  limit cycle(s) can bifurcate from the  anti-saddle positive equilibrium with smaller abscissa (resp.  bigger abscissa).
In addition,  two weak foci can coexist and each of them is order one.
Moreover, it is shown that  one limit cycle can bifurcate from each of the two weak foci simultaneously.
Furthermore, numerical simulations demonstrate that there is also a big stable limit cycle enclosing these two limit cycles.

Although the predator-prey theory has made great progress in the past few decades, there are still a lot of mathematical and ecological problems unsolved. The existence and stability of limit cycles are important topics not only in the global theory for dynamical systems on the plane, but also in the study of mathematical ecology. Such studies have made it possible to be a better understanding of many real world oscillatory phenomena in nature \cite{May1972,Albrecht1973}.

To study the  stability and Hopf bifurcation for  positive equilibria, it  is  better to get the explicit expressions of the coordinates of the  equilibria. However,  the  coordinates for some positive equilibria of  predator-prey systems are generally very complicated, even one can not give exact and explicit expressions of their  coordinates.
This makes it difficult to study the stability and Hopf bifurcation for these  equilibria.
In this paper, we propose some available methods to solve these problems.
By appropriate scalings of the coordinates, any positive equilibrium of system \eqref{dai1.2}
can be transformed into the  equilibrium  $(1, 1)$.
Furthermore, by appropriate time scaling, system \eqref{dai1.2} can be reduced to an equivalent polynomial  system in the interior of the first quadrant.

The rest of the paper is organized as follows. Section $2$ is devoted to some preliminary results including model reduction, analysis of equilibria of an equivalent system of \eqref{dai1.2} and computation of Lyapunov constants at each center or focus type  equilibrium.
The qualitative behavior near the origin is presented in Section $3$,  which is prepared for the global analysis.
In Section $4$,  we investigate the Hopf bifurcation at each weak focus and the global stability.
The Hopf bifurcation and global dynamics when system has a unique non-degenerate positive equilibrium and when system has three distinct positive equilibria are studied in Section $4.1$ and Section $4.2$, respectively. All results are proved by real analysis and symbolic computation. The paper ends with a discussion.
\section{Some preliminary results}

In this paper, we begin to study  system \eqref{dai1.2}. The preliminary results provided in this section will be useful for the proof of the  main  results of the present paper.

\subsection{Model reduction}

To simplify the computation, we reduce system \eqref{dai1.2} to an equivalent polynomial differential system by scalings of the coordinates.

From the analysis in \cite{Li2007}, we know that system \eqref{dai1.2} has at least one and at most three positive equilibria in the interior of the first quadrant. Under the transformation \eqref{dai1.3},
the expressions of some positive equilibria are still too complicated that we can not study them. Motivated by the idea of \cite{Zhao2015}, in this paper, we will use some available methods to solve this problem.

Without loss of generality,  assume that $E(x^{*}, y^{*})$ is an arbitrary positive equilibrium of system \eqref{dai1.2}. By applying the following scalings of the coordinates and the constants scaling
\begin{equation}\label{dai2.1}
\begin{aligned}
\bar{x}=\frac{x}{x^{*}},\,\, \bar{y}=\frac{y}{y^{*}},\,\, \bar{t}=rt,\,\,\bar{K}=\frac{K}{x^{*}}, \,\, \bar{b}=\frac{b}{{x^{*}}^2},\,\, \bar{m}=\frac{m y^{*}}{r {x^{*}}^2},\,\, \bar{s}=\frac{s}{r},\,\, \bar{h}=\frac{h x^{*}}{y^{*}},
\end{aligned}
\end{equation}
and dropping the bars, system \eqref{dai1.2} is reduced to
\begin{eqnarray}\label{dai2.4}
\left\{
\begin{array}{ll}
\dot{x}=x\left(1-\dfrac{x}{K}\right)-
\dfrac{m x y}{x^2+b}, \\[8pt]
\dot{y}=sy\left(1-\dfrac{y}{h x}\right).
\end{array}
\right.
\end{eqnarray}
Noting that equilibrium  $E(x^*,y^*)$ of system \eqref{dai1.2} is reduced to the  equilibrium  $E^{*}(1, 1)$ of system \eqref{dai2.4}, we have $m=(b+1)\left(1-1/K\right)$ and  $h=1$.
That is to say, system \eqref{dai2.4} can be written as
\begin{eqnarray}\label{dai2.5}
\left\{
\begin{array}{ll}
\dot{x}=x\left(1-\dfrac{x}{K}\right)-
\dfrac{(K-1)(b+1)x y}
{K(x^2+b)}, \\[8pt]
\dot{y}=sy\left(1-\dfrac{y}{x}\right).
\end{array}
\right.
\end{eqnarray}
Note that $K>x^{*}$ becomes $K>1$, and $b>0$, $s>0$ remain the same when system \eqref{dai1.2} is transformed into \eqref{dai2.5}. Thus, throughout the rest of this paper, we assume that  the parameters $K$, $b$, $s$ satisfy the following conditions
\begin{equation}\label{cond}
\begin{aligned}
K>1,\quad b>0,\quad s>0.
\end{aligned}
\end{equation}
Since the transformation \eqref{dai2.1} is a linear sign-reserving transformation,  system \eqref{dai2.5} has the same qualitative property as the system \eqref{dai1.2}. Note that system \eqref{dai2.5} is not well-defined at $x=0$.
At the same time, according to the biological meaning of this model, we only need to consider system \eqref{dai2.5} in $\mathbb{R}^+_2=\{(x,  y) : x> 0,\  y\geq0\}$. It's  standard  to  show  that  all solutions of  \eqref{dai2.5} with positive initial values are positive and bounded, and  will eventually tend into the region $\Omega=\{(x(t),  y(t)): 0 <x(t)< K \ \textrm{and} \ 0\leq y(t)< K\}$. Therefore, system \eqref{dai2.5} is well-defined in a subset of $\mathbb{R}^+_2$.

Note that the fact that system \eqref{dai2.5} cannot be linearized at the origin $(0, 0)$
might make system \eqref{dai2.5} have very rich and complicated dynamics, see \cite{Xiao2001} for instance. In addition, the computation of  Lyapunov constants at an equilibrium (see Section 2.3) is generally applied for polynomial differential systems. Therefore, we need to transform system \eqref{dai2.5}  into an equivalent polynomial differential system. By scaling
\begin{equation*}
\begin{aligned}
d\tau=\frac{dt}{K x (x^2+b)},
\end{aligned}
\end{equation*}
(we will still use $t$ to denote $\tau$ for ease of notation) system \eqref{dai2.5} is transformed into  the following quintic polynomial differential system
\begin{eqnarray}\label{dai2.7}
\left\{
\begin{array}{ll}
\dot{x}=(x^2+b)\left(K-x\right)x^2-(K-1)(b+1)x^2y, \\[6pt]
\dot{y}=K s y\left(x-y\right)(x^2+b).
\end{array}
\right.
\end{eqnarray}
Clearly, system \eqref{dai2.7} has the same topological structure as system \eqref{dai2.5} in $\mathbb{R}^+_2$  since $K x (x^2+b)>0$ for  $x>0$. In the following we only need to consider system \eqref{dai2.7} in $\mathbb{R}^+_2$ with parameters satisfying  \eqref{cond}. It's obvious  that  solutions of  \eqref{dai2.7} with positive initial values are positive and bounded and  will eventually tend into the region $\Omega$.

\subsection{Analysis of equilibria  of  system \eqref{dai2.7}}

Clearly, system \eqref{dai2.7} always has a boundary equilibrium $E_0=(K, 0)$ and this equilibrium is a hyperbolic saddle for all parameters satisfying \eqref{cond}. In addition, $E_0$ divides  the positive $x$-axis into two parts which are two stable manifolds of $E_0$ and there exists a unique unstable manifold of $E_0$ in the interior of $\mathbb{R}^2_+$.

Next  we are going to study the positive (i.e., interior) equilibria of system \eqref{dai2.7}. Assume $\tilde{E}(\tilde{x}, \tilde{y})$ is a positive equilibrium of system \eqref{dai2.7}, then $y=\tilde{x}$ and $\tilde{x}$ is a positive root of the equation
\begin{equation}\label{dai2.8}
\begin{aligned}(x-1)\left(x^2-(K-1)x+K b\right)=0
\end{aligned}
\end{equation}
on the interval $(0, K)$.
 The Jacobian matrix at the  positive equilibrium $E^*(1, 1)$ is
$$\textrm{J}(E^*)=\left(\begin{array}{cc}
2 K-b-3 & - K b-K+b+1\\
K s (b+1)  & -K s (b+1)
\end{array}\right).
$$
Then the characteristic polynomial of $\textrm{J}(E^*)$ is
\begin{equation*}
\Theta(\lambda)=\lambda^2-\textrm{Tr}\left(J(E^{*})\right)\lambda+\textrm{Det}\left(J(E^{*})\right),
\end{equation*}
where
\begin{equation}\label{TR}
\textrm{Tr}\left(J(E^{*})\right)=
2\,K-b-3-K \left( b+1 \right)s,\quad
\textrm{Det}\left(J(E^{*})\right)=
K s \left( b+1 \right)  \left( K b-K+2 \right).
\end{equation}
Therefore, $E^*(1, 1)$ is a saddle if $\textrm{Det}\left(J(E^{*})\right)<0$, a degenerate equilibrium if $\textrm{Det}\left(J(E^{*})\right)=0$, and  an elementary equilibrium if $\textrm{Det}\left(J(E^{*})\right)>0$, respectively. If $\textrm{Det}\left(J(E^{*})\right)>0$, i.e.,
$K b-K+2>0$, a straightforward calculation shows that
\begin{equation}\label{dai2.10}
\left(\textrm{Tr}\left(J(E^{*})\right)\right)^2-4 \textrm{Det}\left(J(E^{*})\right)=K^2(b+1)^2s^2-2 K (b+1) (2 K b-b+1)s+(2 K-b-3)^2\triangleq \psi(s).
\end{equation}
Note that $\bar{\Delta}=\left(-2 K (b+1) (2 K b-b+1)\right)^2-4 K^2(b+1)^2 (2 K-b-3)^2=
16\,{K}^{2} ( b+1) ^{3} (K-1) (K b-K+2)>0$.
From \eqref{cond} and \eqref{dai2.10}, $\psi(s)$ has two positive zeros at $s_1$ and $s_2$ with $0<s_1<s_2$.
Furthermore, $\psi(s)\geq 0$ for $s\in(0, s_1]\cup[s_2, +\infty)$ and $\psi(s)<0$ for $s\in(s_1, s_2)$, respectively. Hence, $E^*$ is a node if $s\in(0, s_1)\cup(s_2, +\infty)$ and a focus or center if $s\in(s_1, s_2)$, respectively. The stability of $E^*$ is determined by $\textrm{Tr}\left(J(E^{*})\right)$.

Summarizing the above  discussions, we arrive at the following results.
\begin{lemm}\label{L2.1} Let $s^{*}=(2 K-b-3)/(K (b+1))$ be the zero of $\textrm{Tr}\left(J(E^{*})\right)$ of $s$.
\begin{itemize}
\item[(a)]Suppose $K b-K+2<0$, then $E^{*}(1, 1)$ is a hyperbolic saddle;
\item[(b)] Suppose $K b-K+2=0$, then $E^{*}(1, 1)$ is a degenerate equilibrium;
\item[(c)] Suppose $K b-K+2>0$, then $E^{*}(1, 1)$ is a node if $s\in(0, s_1)\cup(s_2, +\infty)$ and a focus or
center if $s\in(s_1, s_2)$, respectively. More precisely, the following results hold.
\begin{itemize}
\item[(c1)]If $2 K-b-3\leq 0$, then $E^{*}(1, 1)$ is a locally asymptotically stable  node or focus;
\item[(c2)]  Let  $2 K-b-3>0$.
\begin{itemize}
\item[(i)] If  $s>s^{*}$, then $E^{*}(1, 1)$ is a locally asymptotically stable node or focus;
\item[(ii)]  If $s<s^{*}$, then $E^{*}(1, 1)$ is an unstable node or focus;
\item[(iii)]If $s=s^{*}$, then $E^{*}(1, 1)$ is a weak focus or center.
\end{itemize}
\end{itemize}
\end{itemize}
\end{lemm}
\begin{proof}
Under the conditions \eqref{cond}, it follows from Eq. \eqref{TR} and Eq. \eqref{dai2.10} that  the conclusions of Lemma \ref{L2.1} are clearly verified by applying the results in Section II.2 in \cite{ZhangZF1992}.
\end{proof}
The number of positive  equilibria of
system \eqref{dai2.7} is determined by the number of positive roots of Eq. \eqref{dai2.8} on the interval $(0, K)$.
Noting that the value $1$ is always a positive root of Eq. \eqref{dai2.8},
the equation \eqref{dai2.8} can have one, two or
three positive roots in the interval $(0, K)$ which can be evaluated by using the root formula of the quadratic polynomial. Correspondingly, system \eqref{dai2.7} can have one, two, or three positive equilibria.
Applying the results of Lemma \ref{L2.1}, after tedious analysis, we have the following results which are equivalent to Lemma 2.1 of \cite{Li2007}.
\begin{lemm}\label{L2.2}
Let $\overline{\Delta}=(K-1)^2-4Kb$.
\begin{itemize}
\item[(a)]Suppose  $b>(K-1)^2/(4K)$, then system \eqref{dai2.7} has a unique positive equilibrium $E^{*}(1, 1)$, which is an elementary  anti-saddle equilibrium;
\item[(b)]Suppose $b=(K-1)^2/(4K)$.
\begin{itemize}
\item[(b1)] If $K=3$, i.e., $b=1/3$, then system \eqref{dai2.7} has a unique positive equilibrium $E^{*}(1, 1)$,  which is a degenerate equilibrium;
\item[(b2)] If $K\neq 3$, then system \eqref{dai2.7} has two different positive equilibria: a degenerate equilibrium $E_2^{*}((K-1)/2, (K-1)/2)$ and an elementary anti-saddle equilibrium $E^{*}(1, 1)$;
\end{itemize}
\item[(c)] Suppose $b<(K-1)^2/(4K)$.
\begin{itemize}
\item[(c1)] If $b=1-2/K$, then system \eqref{dai2.7} has two different positive equilibria: a degenerate  equilibrium $E^{*}(1, 1)$ and an elementary anti-saddle equilibrium $E_2^{*}(K-2, K-2)$;
\item[(c2)] If $b\neq 1-2/K$, then system \eqref{dai2.7} has three different positive equilibria $E^{*}(1, 1)$, $E_2^{*}((K-1-\sqrt{\overline{\Delta}})/2), (K-1-\sqrt{\overline{\Delta}})/2)$ and $E_3^{*}((K-1+\sqrt{\overline{\Delta}})/2), (K-1+\sqrt{\overline{\Delta}})/2))$,  which are all elementary equilibria. Furthermore,  $E^{*}$ and $E_3^{*}$ are both anti-saddle, and $E_2^{*}$ is a saddle, provided that 1 is the minimum root of  Eq. \eqref{dai2.8}.
\end{itemize}
\end{itemize}
\end{lemm}
The cases (b) and (c1) of Lemma \ref{L2.2}  have been investigated by Li and Xiao \cite{Li2007} and Huang, Xia and Zhang \cite{Huang2016}. In the rest of this  paper  we only need to focus on the two cases (a) and (c2) of Lemma \ref{L2.2}.
Furthermore, noting that the type and stability of $E^{*}(1, 1)$ have been determined for all the cases of Lemma \ref{L2.1} except the subcase (iii) of (c2), we only need to  determine if $E^{*}(1, 1)$ is a center or a weak focus under the corresponding conditions.
\subsection{Computation of Lyapunov constants}
\quad\ \ In this subsection, we are going to study Hopf bifurcation of system \eqref{dai2.7}. Without loss of generality, we suppose that  $E^{*}(1, 1)$ is a center or focus type equilibrium  of system \eqref{dai2.7}, which yields that the condition (iii) in Lemma \ref{L2.1} (c2) holds, i.e.,
\begin{equation}\label{cond2.3}
\begin{aligned}
K b-K+2>0,\quad 2 K-b-3>0,\quad  s=s^{*}.
\end{aligned}
\end{equation}
To determine if $E^{*}(1, 1)$ is a center or a weak focus, one needs to calculate the focal values of system \eqref{dai2.7} at this equilibrium. If $E^{*}(1, 1)$ is a weak focus, we will  determine the direction of Hopf bifurcation and study the number of limit cycles bifurcating from it.
In this paper, we will use the Lyapunov constants instead of the focal values to solve these problems. The equivalence between the Lyapunov constants and the focal values can be seen in \cite{Liu2001} and Chapter $1$ in \cite{Liu2008}, or \cite{Dai2017} for more details.

For convenience, we denote $V_1\triangleq \textrm{Tr}\left(J(E^{*}(1, 1))\right)$. Under the condition: $V_1=0$, i.e., $s=s^{*}$, we will compute the Lyapunov constants of system \eqref{dai2.7} at $E^{*}(1, 1)$. Firstly, shifting $E^{*}(1, 1)$ to the origin by the  transformation $u=x-1$, $v=y-1$, system \eqref{dai2.7} becomes
\begin{eqnarray}\label{dai2.12}
\left\{
\begin{array}{rl}
\dot{u}=&a_{10}u+a_{01}v+a_{20}u^2+a_{11}uv+a_{30}u^3+a_{21}u^2v
+a_{40}u^4-u^5,\\[6pt]
\dot{v}=&b_{10}u+b_{01}v+b_{20}u^2+b_{11}uv+b_{02}v^2+b_{30}u^3+b_{21}u^2v+b_{12}uv^2
+b_{31}u^3v+b_{22}u^2v^2,
\end{array}
\right.
\end{eqnarray}
where
$$a_{10}=2 K-b-3,\ a_{01}=-(K-1)(b+1),\
a_{20}=5 K-2 b-9,\ a_{11}=2a_{01},\ a_{21}=a_{01},\ a_{30}=4 K-10-b,$$
$$a_{40}=K-5, \,
b_{10}=a_{10},\ b_{01}=-a_{10},\,
b_{20}=2 a_{10}/(b+1),\, b_{11}=(b-1)b_{20}/2,\, b_{02}=-a_{10},\, b_{12}=-b_{20},$$
$$b_{30}=b_{21}=b_{31}=-b_{22}=b_{20}/2.$$
Applying the algorithm in \cite{Sang2016} and with the help of the software Maple, we
get  the first two Lyapunov constants of system \eqref{dai2.7} at the equilibrium $E^{*}(1, 1)$ as follows:
\begin{equation}\label{dai2.13}
\begin{aligned}
V_3=\dfrac{(K-1)^2}{4(K b-K+2)}\upsilon_1,\quad \quad
V_5=-\dfrac{(K-1)^3}{48 (K b-K+2)^3(2K-b-3)}\upsilon_2,
\end{aligned}
\end{equation}
where the quantity $V_5$ is reduced w.r.t. the Gr\"{o}bner basis of $\{V_3\}$, $\upsilon_1$ and $\upsilon_2$ are given in the Appendix A.

\section{Asymptotic behavior of system \eqref{dai2.7} near the origin}

Since we only need to consider system \eqref{dai2.7} in $\mathbb{R}^+_2$,
the qualitative behavior near the origin is important to the global  dynamic behaviors. In this section, we study the singularity $(0, 0)$ of system \eqref{dai2.7} and give all possibilities for the orbits of system \eqref{dai2.7} approach to $(0, 0)$ as $t\rightarrow+\infty$ or $t\rightarrow-\infty$ depending on all parameters.  As a consequence, a sufficient condition for the existence of limit cycle is given in Theorem \ref{T3.2}.

Let  $X_2(x, y)+\Phi(x, y)$ and $Y_2(x, y)
+\Psi(x, y)$ be the right-hand side of the first and second equation in system \eqref{dai2.7}, respectively. Among them, $X_2(x, y)=K b {x}^{2}$, $Y_2(x, y)=K b s x y- K b s {y}^{2}$, which are both homogeneous polynomials in $x$ and $y$ of degree $2$, and $\Phi(x, y)=o(r^2)$,  $\Psi(x, y)=o(r^2)$ as $r\rightarrow 0$, where $x=r\cos \theta, \, y=r\sin \theta$. By  Theorems 3.4, 3.7, 3.8 and 3.10 in \cite{ZhangZF1992}, we obtain  the following theorem.
\begin{theo}\label{T3.1}  Let  $S^+_{\delta} (O)=\{(r, \theta):0<r<\delta,\ 0\leq\theta\leq \pi/2\}$ with  $0<\delta\ll 1$. The following statements hold.
\begin{itemize}
\item[(1)] If  $0<s<1$, then
\begin{itemize}
\item[(a)] the positive $x$-axis is a unique orbit of system \eqref{dai2.7} tending to $(0, 0)$ along $\theta=0$ as  $t\rightarrow-\infty$;
\item[(b)] the positive $y$-axis is a unique orbit of system \eqref{dai2.7} tending to $(0, 0)$ along $\theta=\pi/2$ as $t\rightarrow+\infty$.
\end{itemize}
The phase portrait of system \eqref{dai2.7} near $(0, 0)$ is shown in Fig. 3.1. (1).
\item[(2)] If $s=1$, then
\begin{itemize}
\item[(a)] there is an infinite number of orbits of system \eqref{dai2.7} in $S^+_{\delta} (O)$ tending to $(0, 0)$ along $\theta=0$ as $t\rightarrow-\infty$;
\item[(b)]  the positive $y$-axis is a unique orbit of system \eqref{dai2.7} tending to $(0, 0)$ along $\theta=\pi/2$ as $t\rightarrow+\infty$.
\end{itemize}
The phase portrait of system \eqref{dai2.7} near $(0, 0)$ is shown in Fig. 3.1. (2).
\item[(3)] If $s>1$, then
\begin{itemize}
\item[(a)]  there is an infinite number of orbits of system \eqref{dai2.7} in $S^+_{\delta} (O)$ tending to $(0, 0)$ along  $\theta=0$ as $t\rightarrow-\infty$;
\item[(b)]  the positive $y$-axis is a unique orbit of system \eqref{dai2.7} tending to $(0, 0)$ along $\theta=\pi/2$ as
  $t\rightarrow+\infty$;
\item[(c)] there is a unique orbit of system \eqref{dai2.7} tending to $(0, 0)$ along $\theta_{s}=\arctan ((s-1)/s)$as $t\rightarrow-\infty$. And this orbit is a separatrix that divides $S^+_{\delta} (O)$ into two parts. One part is a hyperbolic sector,
and the other part is a parabolic sector.
\end{itemize}
The phase portrait of system \eqref{dai2.7} near $(0, 0)$ is shown in Fig. 3.1. (3).
\end{itemize}
\end{theo}
From Theorem \ref{T3.1}, for any case, every solution $(x(t),  y(t))$ of system \eqref{dai2.7} with positive initial values will eventually be away from the origin (see Fig. 3.1). Note that the $y$-axis is an invariant straight line of system \eqref{dai2.7}. Therefore, for any case, the bound of the region $\Omega$ can be used as the outer boundary of a Poincar\'{e}-Bendixson annular region. Thus, we can get the following results by applying the Poincar\'{e}-Bendixson Theorem \cite{ZhangZF1992}.
\begin{theo}\label{T3.2}
\ If $E^{*}(1, 1)$ is the unique positive equilibrium of system \eqref{dai2.7} and is unstable, then there is at least one stable limit cycle in $\mathbb{R}^2_+$.
\end{theo}
\begin{figure}\begin{center}
{\includegraphics[height=5cm,width=13cm]{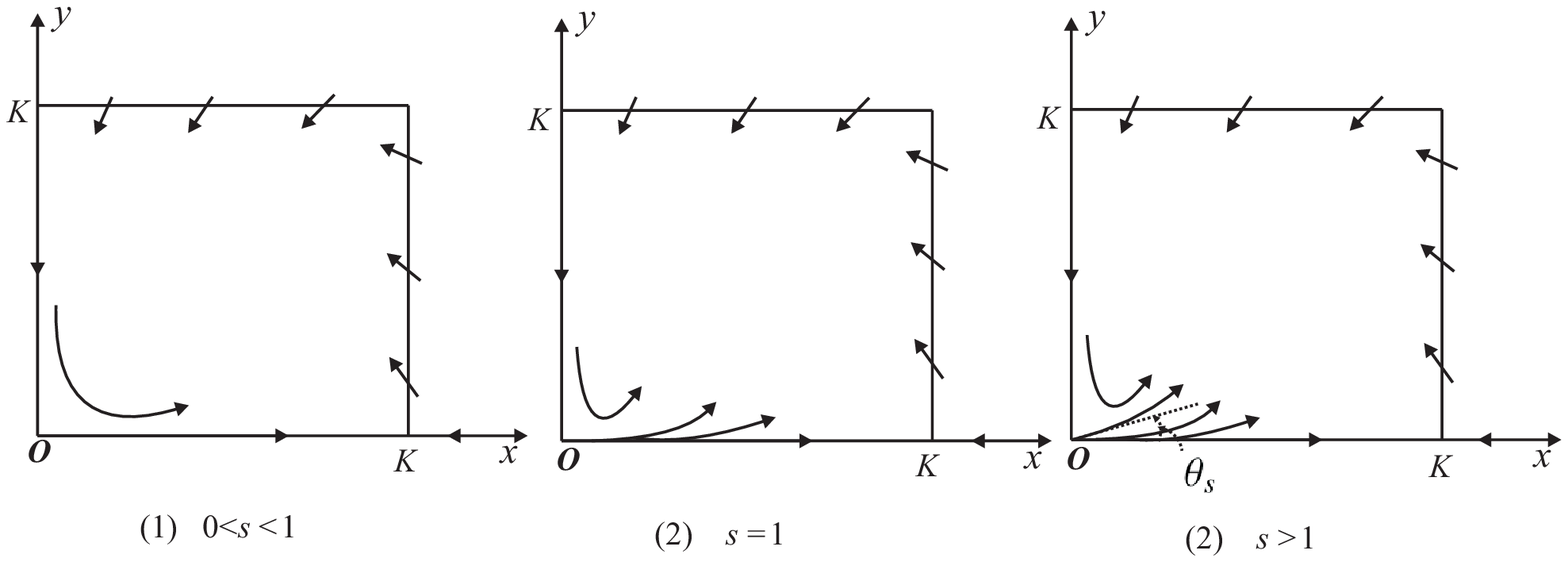}}
\end{center}
{\small {\bf Fig. 3.1.} Vector field on the boundary of $\Omega$ and the topological structure of the orbits of system \eqref{dai2.7} near the origin  for all cases. }
\end{figure}

If  the conditions in   Lemma \ref{L2.2} (a) and  (ii) of Lemma \ref{L2.1} (c2)   hold, then we have \begin{equation}\label{dai3.2}
\begin{aligned}
\frac{(K-1)^2}{4K}<b<2K-3, \quad  s<s^{*}.
\end{aligned}
\end{equation}
The first inequality of \eqref{dai3.2} implies $K>(5+4\sqrt{2})/7$. It follows from Lemmas \ref{L2.1} and \ref{L2.2} that, if the conditions \eqref{dai3.2} hold, then $E^{*}(1, 1)$ is the unique positive equilibrium of system \eqref{dai2.7} which is unstable. From Theorem \ref{T3.2}, we know that system \eqref{dai2.7} has at least one stable limit cycle in $\mathbb{R}^+_2$.
\begin{rema}\label{rem1}
From Theorem \ref{T3.1} and the Poincar\'{e}-Bendixson Theorem \cite{ZhangZF1992}, it's easy to see that the system \eqref{dai2.7} is uniformly persistent in the interior of $\mathbb{R}^+_2$.
\end{rema}

\section{Hopf bifurcation and global dynamics}

In this section, we are going to focus on the two cases listed in  (a) and (c2) of Lemma \ref{L2.2}, i.e., the cases (I) when system \eqref{dai2.7} has a unique non-degenerate positive equilibrium, and (II) when system \eqref{dai2.7} has three distinct positive equilibria. We will consider the Hopf bifurcation and global dynamics of system \eqref{dai2.7} for these two cases.
\subsection{The case when system \eqref{dai2.7} has a unique non-degenerate positive equilibrium}

In this subsection, we first consider system \eqref{dai2.7} with a unique non-degenerate positive equilibrium, which implies the condition in Lemma \ref{L2.2} (a) holds.
The condition $b>(K-1)^2/(4K)$ implies $K b-K+2>0$. From Lemma \ref{L2.1}, it follows that this positive equilibrium is anti-saddle.
We study the Hopf bifurcation and global stability of this equilibrium in the following two subsections, respectively.
\subsubsection{Hopf bifurcation at the unique positive equilibrium}

We first study the Hopf bifurcation at the unique positive equilibrium by applying the Lyapunov constants expressions \eqref{dai2.13}. Assume   that the conditions in (iii) of Lemma \ref{L2.1} (c2) holds.

Note that there are two parameters  $K$ and $b$ involved in  the Lyapunov constants $V_{3}$ and $V_{5}$ (cf.  \eqref{dai2.13}).
We first consider the possibility  whether there exist some values of parameters such that $V_{3}=V_{5}=0$, and hence that the equilibrium $E^{*}(1, 1)$ of system \eqref{dai2.7} is either a center or a weak focus of order $3$ or more.
Moreover, if $E^{*}(1, 1)$ is a weak focus, we will study the number and distribution of limit cycles for  this  system. Our results are as follows.
\begin{theo}\label{T4.1}
If the unique non-degenerate positive equilibrium $E^{*}(1, 1)$ of system \eqref{dai2.7} is center or focus type, then
\begin{itemize}
\item[(1)]it cannot be a center and is a weak  focus of order at most $2$;
\item[(2)] there exist some different parameter values such that system \eqref{dai2.7} has $i$ small limit cycle(s) around $E^{*}(1, 1)$ for $i=1, 2$.
\end{itemize}
\end{theo}

Before proving Theorem \ref{T4.1}, we first give some results which are helpful for the proof of our results. Let $\textbf{K}$ be an algebraically closed field. Given two polynomials $A(x_1, \cdots, x_n), B(x_1, \cdots, x_n)\in \textbf{K}[x_1, \cdots, x_n]$, $(x_1, \cdots, x_n)\in {\textbf{K}}^{n}$, of the forms
\begin{equation*}
A(x_1, \cdots, x_n)=\sum_{i=1}^{k}A_i(x_1, \cdots, x_{n-1})x_n^{i},\quad
B(x_1, \cdots, x_n)=\sum_{i=1}^{l}B_i(x_1, \cdots, x_{n-1})x_n^{i},
\end{equation*}
where both $k$ and $l$ are positive integers.
Denote the Sylvester resultant of $A$ and $B$ with respect to $x_n$, as defined in \cite{GelfandIM1994},  by $\textrm{Res}(A, B, x_n)$.
Then the following lemma holds (see Theorem 5 in \cite{Collins1971}).
\begin{lemm}\label{L4.3}  Denote $C(x_1, \cdots, x_{n-1})\triangleq \textrm{Res}(A, B, x_n)$.
If $(a_1, \cdots, a_{n})$ is a common zero of $A$ and $B$, then $C(a_1, \cdots, a_{n-1})=0$.
Conversely, if $C(a_1, \cdots, a_{n-1})=0$, then at least one of the following statements holds:
\begin{itemize}
\item[(a)] $A_k(a_1, \cdots, a_{n-1})=A_{k-1}(a_1, \cdots, a_{n-1})=\cdots=A_0(a_1, \cdots, a_{n-1})=0$,
\item[(b)] $B_l(a_1, \cdots, a_{n-1})=B_{l-1}(a_1, \cdots, a_{n-1})=\cdots=B_0(a_1, \cdots, a_{n-1})=0$,
\item[(c)] $A_k(a_1, \cdots, a_{n-1})=B_l(a_1, \cdots, a_{n-1})=0$,
\item[(d)] For some $a_{n}\in \textbf{K}$,  $(a_1, \cdots, a_{n})$ is a common zero of $A$ and $B$.
\end{itemize}
\end{lemm}
It's obvious that $C=0$ is a necessary, but not sufficient condition of  $A=B=0$. This fact not only gives a criterion for the existence of common zeros, but also provides a method of finding the common zeros of multivariate polynomial systems.
Let $f_1, \cdots, f_m$ be (finitely many) elements of $\textbf{K}[x_1, \cdots, x_n]$.
Denote the algebraic variety of  $f_1, \cdots, f_m$, the set of common zeros of $f_1, \cdots, f_m$,  by $\textbf{V}(f_1, \cdots, f_m)$.
Then it follows from Lemma \ref{L4.3} that
\begin{equation}\label{dai4.1}
\textbf{V}(A, B)=\textbf{V}(A, B, C).
\end{equation}
This equality provides a method of elimination and will be useful for our main  analysis.

Denote $R^n_+=\{(x_1, \cdots, x_n) : x_i>0,\, i=1, \cdots, n\}$.
For any $n$-variate polynomial $f(x)=f(x_1, \cdots, x_n)$ in $R^n_+$, denote the summation of the positive terms in $f(x)$ and that of the negative terms in $f(x)$ by $f^+(x_1, \cdots, x_n)$ and $f^-(x_1, \cdots, x_n)$, respectively. Obviously, $f=f^++f^-$. Note that both $f^+$ and
$f^-$ are monotone in $R^n_+$. The following lemma is given by  Theorem 2.3 of \cite{Lu2007}.
\begin{lemm}\label{L4.3}
\ For given constants $0< a_{i}\leq b_{i}$ $(i=1, \cdots, n)$,\\
(1)\ \ if $f^+(a_1, a_{2},  \cdots, a_n)+f^-(b_1, b_{2},  \cdots, b_n)>0$, then for
any $x_i\in [a_i, b_i]$ $(i=1, \cdots, n)$,
$$f(x_1, x_{2},  \cdots, x_n)>0;$$
(2)\ \ if $f^+(b_1, b_{2},  \cdots, b_n)+f^-(a_1, a_{2},  \cdots, a_n)<0$, then for
any $x_i\in [a_i, b_i]$ $(i=1, \cdots, n)$,
$$f(x_1, x_{2},  \cdots, x_n)<0.$$
\end{lemm}
This lemma can be used to determine the sign of a multivariate polynomial over an interval, especially over the real root isolation interval.

Now we use the equality \eqref{dai4.1} and Lemma \ref{L4.3} to prove the following results.
\begin{lemm}\label{L4.4} If the conditions of Theorem \ref{T4.1} hold,
then the first two Lyapunov constants of system \eqref{dai2.7} at $E^{*}(1, 1)$
have no common real root.
\end{lemm}
\begin{proof}
From Lemmas \ref{L2.1} and \ref{L2.2}, we know that if the conditions in Theorem \ref{T4.1} hold, then we have the following conditions
\begin{equation}\label{cond4.1}
\frac{(K-1)^2}{4K}<b<2K-3, \quad K>\frac{5+4\sqrt{2}}{7},\quad K b-K+2>0,\quad s=s^{*}.
\end{equation}
From \eqref{dai2.13}, it suffices to prove the two  polynomials $\upsilon_1$ and $\upsilon_2$  have no common real root satisfying  \eqref{cond} and \eqref{cond4.1}.
We consider these two polynomials in the ring $\textbf{R}[K, b]$.

Calculating the resultant of $\upsilon_1$ and $\upsilon_2$  with respect to $b$ by Maple, we get
\begin{eqnarray}\label{dai4.2}
\upsilon_{12}:=\textrm{Res}(\upsilon_1, \upsilon_2, b)=-3221225472 K^3(K-1)^{13}(K-3)^2\phi_1\phi^2_2\phi_3,
\end{eqnarray}
where $\phi_1=K^3-6 K^2+9 K-3$, $\phi_2=K^2-4 K+1$, and
$\phi_3$ is a polynomial in $K$ of degree $13$.  And more, $\upsilon_{12}$ is well factored over the rational field.
It follows from \eqref{dai4.1} that $\textbf{V}(\upsilon_1, \upsilon_2)=\textbf{V}(\upsilon_1, \upsilon_2, \upsilon_{12})$.
From \eqref{dai4.2}, we have to discuss the following four  cases.\vspace{3mm}

\textbf{Case} (i): $K=3$. In this case,  $\upsilon_1=-(3 b-1) (b^3-3 b^2+15 b+3)$. By Sturm's Theorem,  $\upsilon_1$ has a unique positive root $b=1/3$, which implies  $Kb-K+2=0$. This  contradicts to  \eqref{cond4.1}. Thus,  $\textbf{V}(\upsilon_1, \upsilon_2)=\emptyset$.\vspace{3mm}

\textbf{Case} (ii): $\phi_1=0$. Denote by $S_1$  a semi-algebraic system whose polynomial equations, non-negative polynomial inequalities, positive polynomial inequalities and polynomial inequations are given by $F:=[\phi_{1}, \upsilon_1]$, $N:=[\ ]$, $P:=[4 K b-(K-1)^2, K b-K+2, 2K-b-3]$, and $H:=[\ ]$, respectively, where $[\ ]$ represents the null set.
Clearly, the regular chain $\{\phi_{1}, \upsilon_1\}$ is squarefree.
In addition, by the \textit{IsZeroDimensional} command in Maple, we know that $\{\phi_{1}, \upsilon_1\}$ is zero dimensional. Using the \textit{RealRootIsolate} program in Maple with accuracy $1/10^{10}$, we find that $S_1$ has a unique class of real root $(K_1, b_1)$ satisfying
$$
\begin{scriptsize}
\begin{aligned}
(K_1, b_1)\in\Bigg[{\frac {16661832741}{4294967296}},{\frac {66647330965}{17179869184}}\Bigg]\times\Bigg[{\frac {14810494337133}{4398046511104}},{\frac {59241977348533}{
17592186044416}}
\Bigg]\triangleq \Big[\underline{K}_{1}, \overline{K}_{1}\Big]\times \Big[\underline{b}_{1}, \overline{b}_{1}\Big].
\end{aligned}
\end{scriptsize}
$$
This implies $\upsilon_1(K_1, b_1)=0$. Next we only need to verify whether $(K_1, b_1)$ is a zero of $\upsilon_2$.
From Lemma \ref{L4.3}, we have
$$\upsilon_2(K_1, b_1)\geq {\upsilon}^+_{2}(\underline{K}_{1}, \underline{b}_{1})+{\upsilon}^-_{2}(\overline{K}_{1}, \overline{b}_{1})\approx 8.737820385\times 10^{7}>0,$$which means that $(K_1, b_1)$ is not a common real root of $\{\upsilon_1, \upsilon_2\}$. Therefore, $\textbf{V}(\upsilon_1, \upsilon_2)=\emptyset$.\vspace{3mm}
\begin{figure}
\begin{minipage}[t]{0.5\linewidth}
\centering
\includegraphics[width=2.8in]{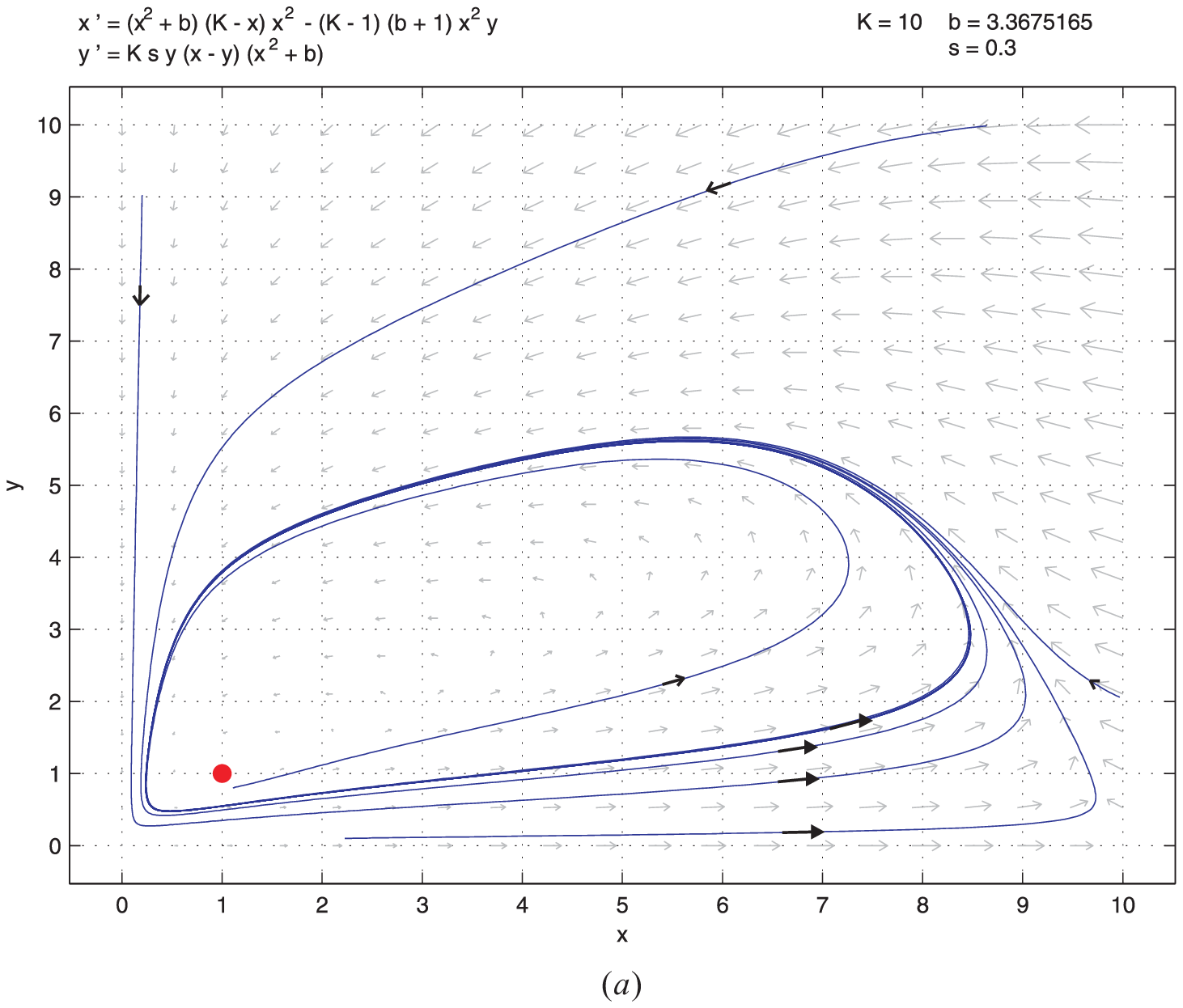}
\end{minipage}%
\begin{minipage}[t]{0.5\linewidth}
\centering
\includegraphics[width=2.8in]{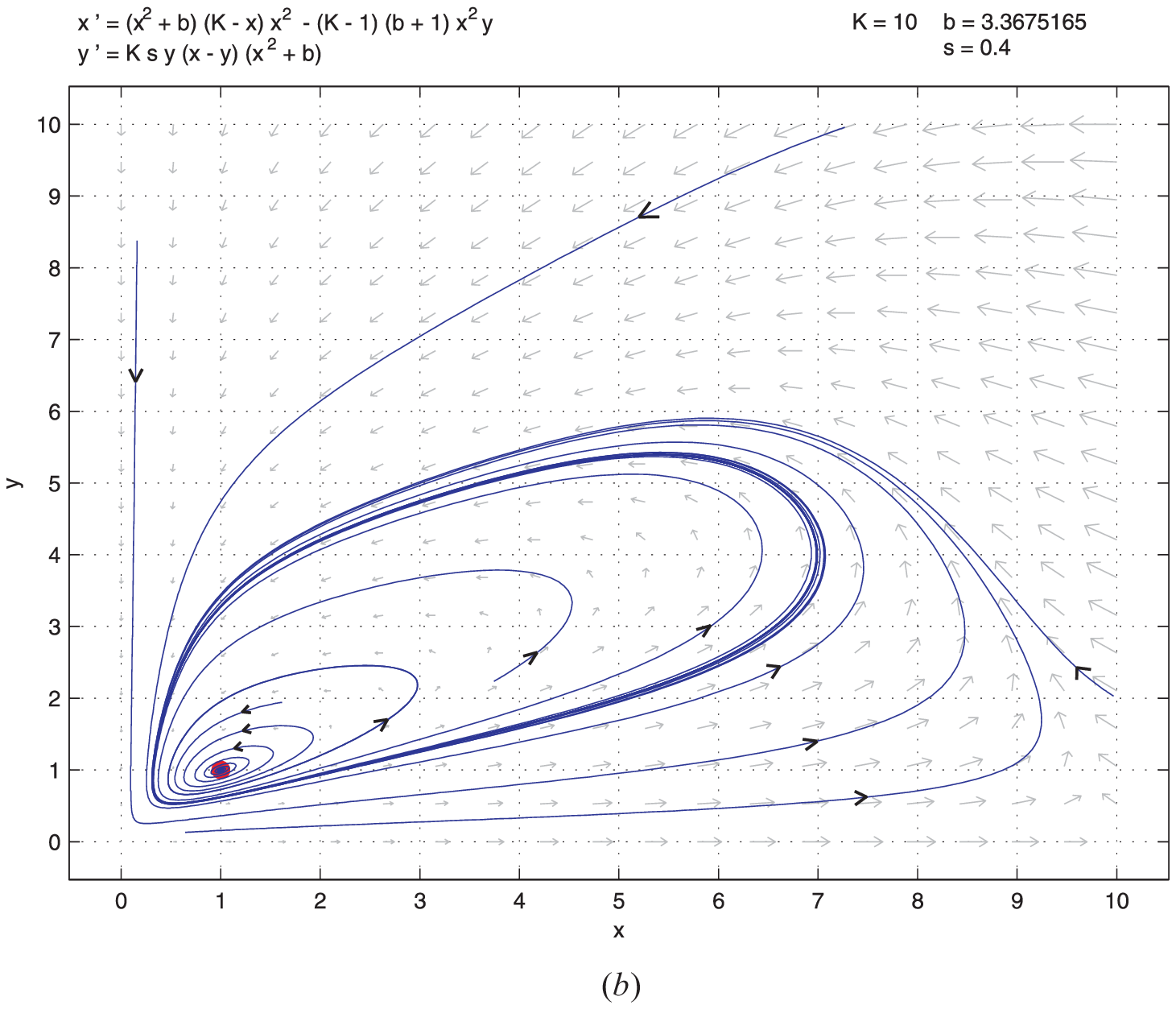}
\end{minipage}
{\small {\bf Fig. 4.1.} Phase portraits for system \eqref{dai2.7} with $b=3.3675165$ near $b_1\approx 3.3675165$: (a) $(K, s)=(10, 0.3)$, a global asymptotically stable limit cycle enclosing an unstable focus;
(b) $(K, s)=(10, 0.4)$, two limit cycles enclosing a stable focus.}
\end{figure}

\textbf{Case} (iii): $\phi_2=0$. Using the same arguments as above and applying the \textit{RealRootIsolate} program in Maple with accuracy $1/10^{10}$, we conclude that $\{\phi_{2}, \upsilon_1\}$ has no common real root satisfying \eqref{cond} and \eqref{cond4.1}, which means $\textbf{V}(\phi_{2}, \upsilon_1)=\emptyset$.
Therefore,  $\textbf{V}(\upsilon_1, \upsilon_2)=\textbf{V}(\upsilon_1, \upsilon_2, \phi_{2})=\emptyset$.\vspace{3mm}

\textbf{Case} (iv): $\phi_3=0$. Similar to the case (ii), using the \textit{RealRootIsolate} program in Maple with accuracy $1/10^{10}$, we know that $\{\phi_{3}, \upsilon_1\}$ has a unique real root $(K_2, b_2)$  satisfying  \eqref{cond}, \eqref{cond4.1} and
$$
\begin{scriptsize}
\begin{aligned}
(K_2, b_2)\in\Bigg[{\frac {50539866915}{17179869184}},{\frac {101079733833}{
34359738368}}\Bigg]\times\Bigg[{\frac {25270463046039627151759819}{
77371252455336267181195264}},{\frac {101081852184158508607039277}{
309485009821345068724781056}}\Bigg].
\end{aligned}
\end{scriptsize}
$$
Similar to determining the sign of $\upsilon_2(K_1, b_1)$, we can check by Maple that $\upsilon_2(K_2, b_2)>0$. This implies $(K_2, b_2)$ is not a common real root of $\{\upsilon_1, \upsilon_2\}$.
Therefore,  $\textbf{V}(\upsilon_1, \upsilon_2)=\textbf{V}(\upsilon_1, \upsilon_2, \phi_{3} )=\emptyset$.

To sum up,  $\upsilon_1$ and $\upsilon_2$ have no common real root satisfying \eqref{cond} and \eqref{cond4.1}. This completes the proof.
\end{proof}
Next we will complete the proof of Theorem \ref{T4.1} by using this Lemma.
\begin{proof}[Proof of Theorem \ref{T4.1}]
From Lemma \ref{L4.4}, if $V_3$ is zero at one class of parameter values, then $V_5$ must not be zero at it.
This implies that  $E^{*}(1, 1)$ cannot be a center and is a weak focus of order  at most $2$.

On the other hand, it follows from the case (ii) in the proof of Lemma \ref{L4.4} that
$\upsilon_1(K_1, b_1)=0$ and $\upsilon_2(K_1, b_1)>0$. From \eqref{dai2.13}, we have $V_3(K_1, b_1)=0$ and $V_5(K_1, b_1)<0$. We first perturb $K_1$ or $b_1$ such that $V_3V_5<0$ and  adjust $s$ such that $V_1=0$.  Then the first  limit cycle bifurcates. The second limit cycle is obtained by  perturbing  $s$ such that  $V_1V_3<0$. Therefore, two limit cycles can bifurcate from $E^{*}(1, 1)$ (see Fig. 4.1). This completes the proof.
\end{proof}
\begin{rema}\label{rem2}
When the system \eqref{dai2.7}  has a unique positive equilibrium, there exist parameter values such that the system has two limit cycles around it. This phenomenon  has been observed not only by Li and Xiao \cite{Li2007} through subcritical Hopf bifurcation and numerical simulations,  but also by Huang et al. \cite{Huang2016} through Bogdanov-Takens bifurcation of codimension $3$ and numerical simulations. It is pointed out that we rigorously prove the  existence of this phenomenon through degenerate Hopf bifurcation.
\end{rema}

\subsubsection{Global stability of the unique positive equilibrium}

In the subsection, we present some sufficient conditions for the global stability of the unique positive equilibrium $E^{*}(1, 1)$, provided that it is locally asymptotically stable. Note that system \eqref{dai2.5} and system \eqref{dai2.7} have the same orbit structure. To simplify the calculation, we consider the system \eqref{dai2.5}. Throughout  the rest of this subsection, we denote
\begin{equation*}
p(x)\triangleq \frac{(K-1)(b+1) x}{K (x^2+b)},\quad h(x)\triangleq \frac{x(K-x)}{K p(x)}.
\end{equation*}

We first give a criterion for the global stability of the unique positive equilibrium.
\begin{lemm}\label{L4.6}  If
\begin{equation}\label{cond5.1}
(x-1)\left(h(x)-1\right)<0 \quad  \textrm{for}\ \  0<x<K,\ x\neq 1,
\end{equation}
then the solutions of system \eqref{dai2.5} with positive initial values satisfy
\begin{equation}\label{dai5.2}
\lim_{t\rightarrow\infty} x(t)=1\quad \textrm{and} \quad \lim_{t\rightarrow\infty} y(t)=1.
\end{equation}
\end{lemm}
\begin{proof}
Construct the following Liapunov function:
\begin{equation*}
V(x, y)=\int_{1}^{x}\frac{\xi-1}{\xi p(\xi)}d\xi+\frac{1}{s}\int_{1}^{y}\frac{\eta-1}{\eta}d\eta,
\end{equation*}
which is non-negative and vanishes only at $E^{*}(1, 1)$.
The time derivative of $V$ calculated along the solutions of system \eqref{dai2.5} is
$$
\begin{aligned}
\dot{V}&=\frac{x-1}{x}\left(\frac{x(K-x)}{K p(x)}-1\right)+\frac{x-1}{x}(1-y)+
(y-1)\frac{(x-1)-(y-1)}{x}\\
&=\frac{x-1}{x}\left(h(x)-1\right)-\frac{(y-1)^2}{x}\\
&=\frac{1}{x}\left[(x-1)\left(h(x)-1\right)-(y-1)^2\right]\leq 0
\end{aligned}
$$
for $x, y>0$.
Note that solutions of system \eqref{dai2.5} with positive initial values are positive and bounded and  will eventually tend into the region $\Omega$. Furthermore, the condition \eqref{cond5.1} implies $E^{*}(1, 1)$ is the unique positive equilibrium  of system \eqref{dai2.5}.  It follows from LaSalle's invariance principle \cite{LaSalle1961} that \eqref{dai5.2} holds.
\end{proof}
The condition \eqref{cond5.1} says that if the horizontal line $y=1$ divides the prey
isocline $y=h(x)$ into two disjoint parts, then $E^{*}(1, 1)$ is globally asymptotically stable in  $\mathbb{R}^+_2$. In particular, if the prey isocline is non-increasing on $0<x<K$, then $E^{*}(1, 1)$ is globally asymptotically stable. Furthermore, we can get some concrete sufficient conditions for global stability by using  $h(x)$.
\begin{theo}\label{T4.2}  If $b>(K-1)(K+3)/4$ or $K=b=3$, then $E^{*}(1, 1)$ is globally asymptotically stable in the interior of $\mathbb{R}^2_+$.
\end{theo}
\begin{proof}
An easy computation yields
$$(x-1) (h(x)-1)=\frac{(x-1)^2 Q(x)}{(K-1) (b+1)},$$
where $Q(x)=-x^2+(K-1)x+K-b-1$.  Therefore, the condition \eqref{cond5.1} is equivalent to
\begin{equation}\label{hx}
Q(x)<0, \quad  \textrm{for}\ \  0<x<K,\ x\neq 1.
\end{equation}
Note that the symmetric axis of the quadratic function $Q(x)$ is $x=(K-1)/2\in(0, K)$ and
$\widetilde{\Delta}=(K-1)^2+4 (K-b-1)=(K-1)(K+3)-4b$.
 The results follows from Lemma \ref{L4.6}.
\end{proof}

To find some other sufficient conditions for global stability, we now suppose $b>(K-1)^2/(4K)$. Then it follows from Lemma \ref{L2.2} that system \eqref{dai2.5} possesses a unique positive equilibrium $\widetilde{E}(\tilde{x}, \tilde{y})$, where
\begin{equation*}
\tilde{y}=\frac{\tilde{x}(K-\tilde{x})}{K p(\tilde{x})}=\tilde{x}.
\end{equation*}
The variational matrix of system \eqref{dai2.5} at $\widetilde{E}(\tilde{x}, \tilde{y})$ takes the form
$$\textrm{J}(\tilde{E})=\left(\begin{array}{cc}
1-\dfrac{2\tilde{x}}{K}-p'(\tilde{x})\tilde{y} & p(\tilde{x})\\
-\dfrac{s\tilde{y}^2}{\tilde{x}^2}   & s-\dfrac{2s\tilde{y}}{\tilde{x}}
\end{array}\right)=\left(\begin{array}{cc}
1-\dfrac{2\tilde{x}}{K}-\dfrac{\tilde{x}(K-\tilde{x})p'(\tilde{x})}{K p(\tilde{x})} & p(\tilde{x})\\
-s   & -s
\end{array}\right).
$$
Thus, $\textrm{Det}\left(J(\tilde{E}(\tilde{x}, \tilde{y}))\right)=s \tilde{x} \left(3\tilde{x}^2-2 K \tilde{x}+K b+K-1\right)/(K(\tilde{x}^2+b))$ and
\begin{equation}\label{Tr}
\textrm{Tr}\left(J(\tilde{E}(\tilde{x}, \tilde{y}))\right)=1-s-\frac{2\tilde{x}}{K}-\frac{\tilde{x}(K-\tilde{x})p'(\tilde{x})}{K p(\tilde{x})}=\frac{-P(\tilde{x})}{K(\tilde{x}^2+b)},
\end{equation}
where
\begin{equation}\label{dai5.4}
P(x)=3x^3+K(s-2)x^2+b x+K b s.
\end{equation}
Note that $b>(K-1)^2/(4K)$ implies that $\textrm{Det}\left(J(E^*(1, 1))\right)=s(K b-K+2)/(K(b+1))>0$. It follows from Lemma \ref{L2.2} that $E^{*}(1, 1)$ is the unique positive equilibrium of system \eqref{dai2.5} and is anti-saddle.
Furthermore,  $P(1)=-K(b+1)\textrm{Tr}\left(J(E^{*}(1, 1))\right)$. Thus,  $E^{*}(1, 1)$ is locally asymptotically stable if $P(1)>0$, and is an unstable node or focus if $P(1)<0$, respectively.

Our basic hypothesis is $b>(K-1)^2/(4K)$ and $P(1)>0$, which implies that the unique positive equilibrium $E^{*}(1, 1)$ is locally asymptotically stable. We divide the condition $P(1)>0$ into two cases.\vspace{3mm}

Case 1. \ \ $P(x)\geq 0$ for all $0<x<K$. \vspace{2mm}

From \eqref{dai5.4}, it's obvious that $P(x)>0$ for all $0<x<K$ if $s\geq 2$.

Let $0<s<2$. Note that $P(0)=K b s>0$, $P(K)=K(K^2+b)(s+1)>0$ and $P'(x)=9x^2+2K(s-2)x+b$.
If $D=K^2(2-s)^2-9 b\leq 0$, i.e., $b\geq K^2(2-s)^2/9$, then $P'(x)\geq 0$ for $x>0$, and hence $P(x)\geq 0$ for all $0<x<K$.
If $D>0$, i.e., $b<K^2(2-s)^2/9$, then from $P'(0)=b>0$ and $P'(K)=K^2(2s+5)+b>0$ it follows that $P'(x)=0$ has two positive roots $0<\alpha_1<\alpha_2<K$, where $\alpha_1= \left(K(2-s)-\sqrt{D}\right)/9,\  \alpha_2=\left(K(2-s)+\sqrt{D}\right)/9$.
It's obvious that $P(x)\geq 0$ for all $0<x<K$ if $P(\alpha_2)\geq 0$.
Next we will express this condition explicitly.
Applying  pseudo division and by command \emph{prem} in Maple, we get
\begin{equation}\label{R}
27 P(x)=(9 x+K s-2 K) P'(x)+R(x),
\end{equation}
where $R(x)= \left( -2\,{K}^{2}{s}^{2}+8\,{K}^{2}s-8\,{K}^{2}+18\,b \right) x+26\,
Kbs+2\,Kb$. It follows from \eqref{R} and $P'(\alpha_2)=0$ that $P(\alpha_2)=R(\alpha_2)/27$.
A direct calculation yields
\begin{equation}\label{Ra2}
R(\alpha_2)=\frac{2}{243}\left(K \Psi-D^{\frac{3}{2}}\right),
\end{equation}
where $\Psi=(108\,s+27) b-{K}^{2} (2-s)^{3}$.  From \eqref{Ra2}, $P(\alpha_2)\geq 0$ is equivalent to
\begin{equation}\label{b1}
b>\frac{{K}^{2} (2-s)^{3}}{108\,s+27}\triangleq b_1^*,\,\,\,\,{\rm{and}}\,\,\,\,K^2 \Psi^2-D^3=243 b(3b^2+\chi_1 b+\chi_0)\geq 0,
\end{equation}
where $\chi_1=K^2( 47s^2+28s-1)$, $\chi_0=K^4s ( s-2)^3<0$. Then $K^2 \Psi^2-D^3\geq 0$ implies
\begin{equation}\label{b2}
b\geq \dfrac{-( 47\,{s}^{2}+28\,s-1) +\sqrt{(s+1)(13 s+1)^3}}{6}\cdot K^2 \triangleq b_2^*.
\end{equation}
Since $0<s<2$, it's not difficult to verify that $K^2(2-s)^2/9>b_2^*>b_1^*$. Thus, $P(x)\geq 0$ for  $0<x<K$  if  $0<s<2$ and $b\geq b_2^*$. \vspace{3mm}

Case 2.\ \ $P(x)$ changes its sign in $(0,K)$. \vspace{2mm}

If $P(x)$ changes its sign in $(0,K)$, then it follows from  the discussion in Case 1 that  $0<s<2$ and $b<b_2^*$.  Then $P(x)=0$ has two positive roots $\beta_1$, $\beta_2$ with $0<\alpha_1<\beta_1<\alpha_2<\beta_2<K$.  And $P(x)$ can be written as
$P(x)=3(x+\beta_3)(x-\beta_1)(x-\beta_2)$ with $\beta_3>0$. Thus, the local asymptotical stability condition $P(1)>0$ can be formulated as
\begin{equation}\label{cond4.12}
\beta_1>1\quad \rm{or} \quad \beta_2<1.
\end{equation}
Unfortunately, we can not  get  the global stability  for this case.

We now state our main results for the global stability of $E^{*}(1, 1)$.
\begin{theo}\label{T4.3} Suppose  $b>(K-1)^2/(4 K)$.  The equilibrium $E^{*}(1, 1)$ is globally asymptotically stable in the interior of $\mathbb{R}^2_+$ if one of the following two conditions holds: (i) $s\geq 2$, (ii) $0<s<2$ and $b\geq b_2^*$.
\end{theo}
\begin{proof}
From the above discussions,  $E^{*}(1, 1)$ is the unique positive equilibrium of system \eqref{dai2.5} and is locally asymptotically stable. If system  \eqref{dai2.5} has no closed orbits, then $E^{*}(1, 1)$ is globally asymptotically stable by the  Poincar\'{e}-Bendixson theorem. This can be proved  by the Dulac criterion. Construct the Dulac function
\begin{equation*}
B(x, y)=\frac{1}{p(x)}\cdot\frac{1}{y^2}, \quad x>0,\quad y>0.
\end{equation*}
Denote  the right-hand side of the first and second equation in system \eqref{dai2.5} by $X(x, y)$ and $Y(x, y)$, respectively.
Then from system \eqref{dai2.5} and the hypothesis in this theorem,
\begin{equation}\label{dai5.7}
\Delta=\frac{\partial( B X)}{\partial x}+\frac{\partial( B Y)}{\partial y}=-\frac{B(x, y)P(x)}{K(x^2+b)}\leq 0.
\end{equation}
Thus, there are no nontrivial periodic solutions and we complete the proof.
\end{proof}

\subsection{The case when system \eqref{dai2.7} has three distinct positive equilibria}

This subsection is devoted to study  system \eqref{dai2.7} with three distinct positive equilibria, which means the condition in (c2) of  Lemma \ref{L2.2}   holds.
Denote these three positive equilibria by $E^{*}_1(x_1^{*}, y_1^{*})$, $E^{*}_2(x_2^{*}, y_2^{*})$ and $E^{*}_3(x_3^{*}, y_3^{*})$ with $x_1^{*}<x_2^{*}<x_3^{*}$, respectively, called the \emph{first}, \emph{second} and \emph{third}
positive equilibria respectively. Then it follows from Lemma \ref{L2.2} that the second positive equilibrium is a saddle and the others are both anti-saddle.
We will study the Hopf bifurcation and global dynamics in $\mathbb{R}^+_2$.
This problem can be further split into two cases.
That is, the case when one of the two anti-saddle positive equilibria is center or focus type and the case when both anti-saddle positive equilibria are center or focus type.
For each case, we will investigate the type and stability of each positive equilibrium, Hopf bifurcation at each weak focus, and the number and distribution of limit cycles. We state them in the following two subsections.

\subsubsection{Hopf bifurcation at  each  center or focus type equilibrium}

We first consider the case when one of the two anti-saddle equilibria is center or focus type.  We will investigate the type and stability of this equilibrium. If it is a weak focus, we will study the maximal number of limit cycles bifurcating from it by applying the Hopf bifurcation theory. Our results are as follows.
\begin{theo}\label{T4.4} If system \eqref{dai2.7} has three distinct positive equilibria, then  the following statements hold:
\begin{itemize}
\item[(1)] If the first positive equilibrium is center or focus type, then it cannot be a center and is a weak focus of order at most $2$. Furthermore,  there exist parameter values such that system \eqref{dai2.7} has $i$ small limit cycle(s) around it for  $i=1, 2$.
\item[(2)] If the third positive equilibrium is center or focus type, then it cannot be a center and is an unstable multiple focus with multiplicity one. Furthermore, there exist parameter values such that  system \eqref{dai2.7} has one small stable limit cycle around it.
\end{itemize}
\end{theo}
\begin{proof}
Suppose system \eqref{dai2.7} has three distinct positive equilibria, then the condition in  (c2) of Lemma \ref{L2.2} holds, i.e.,
\begin{equation}\label{cond4.3}
b<\frac{(K-1)^2}{4K},\quad
b\neq 1-\dfrac{2}{K}.
\end{equation}

(1)  We first prove the first part of Theorem \ref{T4.4}.
Without loss of generality, we assume $E^{*}(1, 1)$ is the first positive equilibrium. This implies the value
$1$ is the minimal root of  Eq. \eqref{dai2.8} in the interval $(0, K)$ and hence we have  $(K-1-\sqrt{(K-1)^2-4Kb})/2>1$, which implies
\begin{equation}\label{cond4.4}
K>3, \quad  K b-K+2>0.
\end{equation}
It follows from \eqref{cond4.4} and Lemma \ref{L2.1} that $E^{*}(1, 1)$ is an anti-saddle elementary equilibrium.
Furthermore, if $E^{*}(1, 1)$ is center or focus type, then the condition \eqref{cond2.3} hold.

To prove the first part of the assertion (1), we only need to prove that the first two Lyapunov constants  in \eqref{dai2.13}  have no common real root, i.e.,
 $\upsilon_1$ and  $\upsilon_2$ have no common real root,  satisfying \eqref{cond}, \eqref{cond2.3}, \eqref{cond4.3} and \eqref{cond4.4}. The conditions  \eqref{cond}, \eqref{cond2.3}, \eqref{cond4.3} and \eqref{cond4.4}  imply
\begin{equation}\label{cond4.5}
K>3, \quad 1-\frac{2}{K}<b<\textup{min}\bigg\{\frac{(K-1)^2}{4K},\ 2 K-3\bigg\}.
\end{equation}
Denote by $S_2$ a semi-algebraic system whose polynomial equations, non-negative polynomial inequalities, positive polynomial inequalities and polynomial inequations are given by $F:=[ \upsilon_1, \upsilon_2]$, $N:=[\ ]$, $P:=[K-3, K b-K+2, 2K-b-3, (K-1)^2-4 K b]$, and $H:=[\ ]$, respectively, where $[\ ]$ represents the null set.
Obviously, the regular chain $\{\upsilon_1, \upsilon_2\}$ is squarefree.
In addition, by the \textit{IsZeroDimensional} command in Maple, we know that $\{\upsilon_1, \upsilon_2\}$ is zero dimensional.
Using the \textit{RealRootIsolate} program in Maple with accuracy $1/10^{20}$,
we conclude that $\{\upsilon_1, \upsilon_2\}$ has no common real root satisfying \eqref{cond4.5}.
 Therefore, $E^{*}(1, 1)$ can not be  a center and  is a weak focus of order  at most $2$.

Next we will prove that there exist at least $2$ limit cycles around $E^{*}(1, 1)$. To this end,  we only need to find some parameter values such that two limit cycles can appear near $E^{*}(1, 1)$.  Using the  same arguments as above and  with the  accuracy $1/10^{20}$, we can find $\upsilon_1$ and $\phi_1$  have a unique common real root $(K_3, b_3)$ satisfying \eqref{cond4.5}, where $\phi_1$ is given in \eqref{dai4.2}, and
$$
\begin{scriptsize}
\begin{aligned}
(K_3, b_3)\in\Bigg[{\frac {572496213716809665905}{147573952589676412928}}, {\frac {
286248106858404832953}{73786976294838206464}}
\Bigg]\times\Bigg[{\frac {
19026100882440877121159}{37778931862957161709568}}, {\frac {
2378262610305109640145}{4722366482869645213696}}
\Bigg].
\end{aligned}
\end{scriptsize}
$$
Similar to determining the sign of $\upsilon_2(K_1, b_1)$ in the case (ii) of the proof of Lemma \ref{L4.4}, we can check by Maple that $\upsilon_2(K_3, b_3)>0$.
It follows from \eqref{dai2.13} that we have $V_3(K_3, b_3)=0$ and $V_5(K_3, b_3)<0$.
We first perturb $K$ near $K_3$ or $b$ near $b_3$ such that $V_3V_5<0$ and adjust $s$ such that $V_1=0$ holds.  The first  limit cycle bifurcates. The second limit cycle is obtained by  perturbing  $s$ such that $V_1$ has  the different  sign with  $V_3$.  This proves the first part of Theorem \ref{T4.4}.
\begin{figure}
\begin{minipage}[t]{0.5\linewidth}
\centering
\includegraphics[width=2.8in]{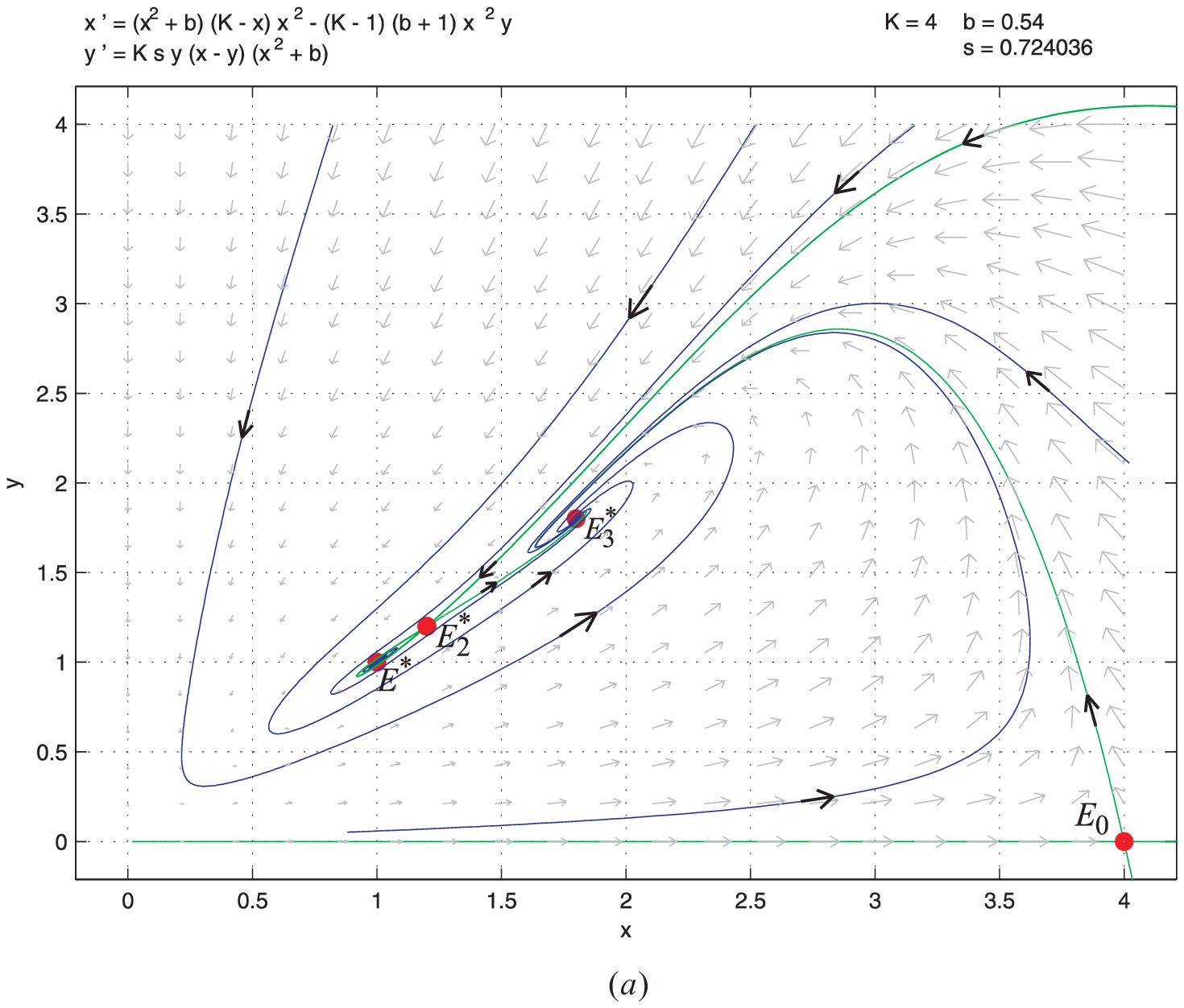}
\end{minipage}%
\begin{minipage}[t]{0.5\linewidth}
\centering
\includegraphics[width=2.8in]{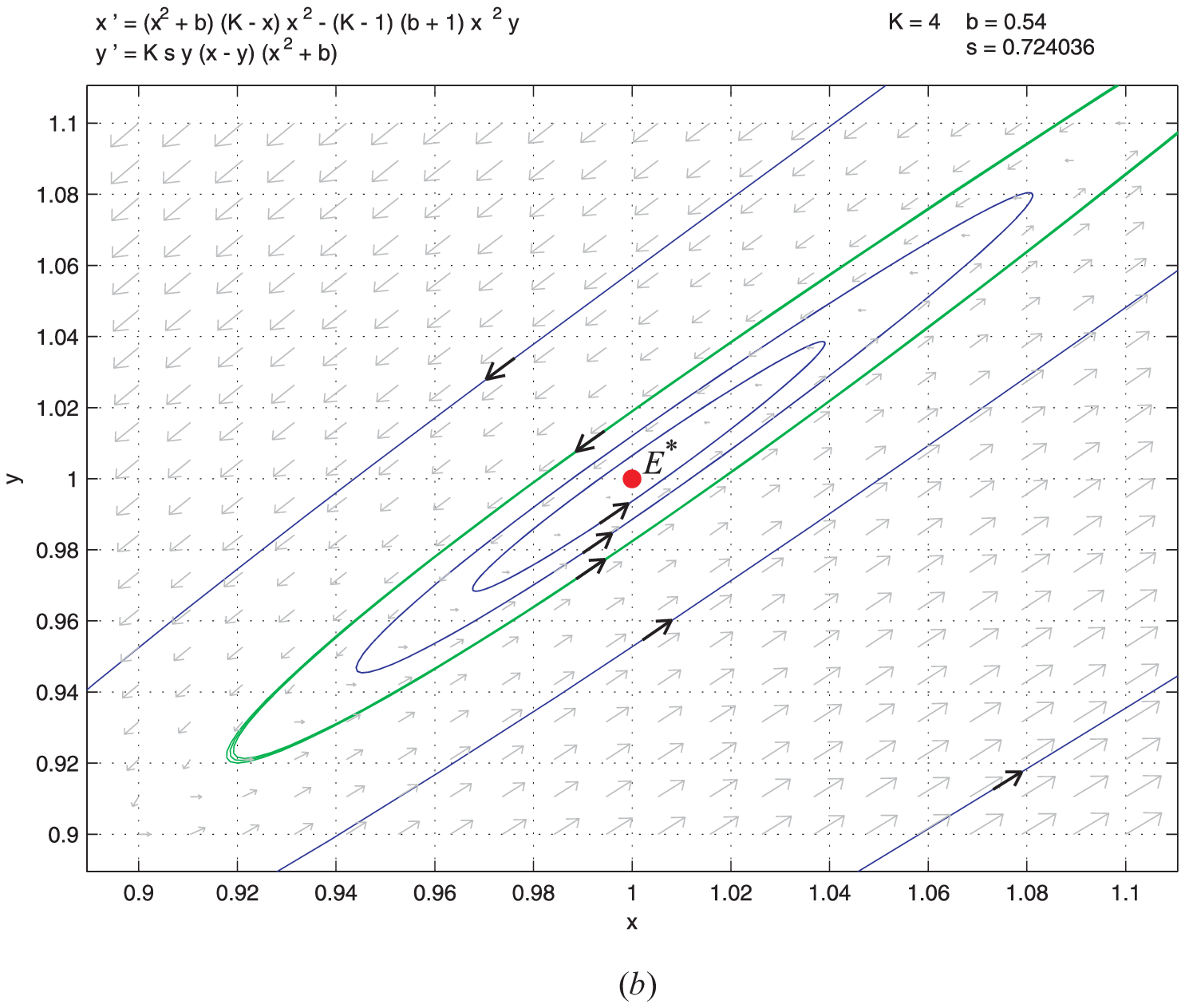}
\end{minipage}
{\small {\bf Fig. 4.2.}
Phase portraits for system \eqref{dai2.7} with $(K, b)=(4, 0.54)$ near $(K_3, b_3)\approx(3.879385, 0.503617)$ and $s=0.724036$, which satisfy the conditions \eqref{cond4.5}: (a) two limit cycles enclosing the first positive equilibrium $E^{*}(1, 1)$; (b) the magnification of the region near $E^{*}(1, 1)$ in the left figure.}
\end{figure}

(2) Now we are going to prove the second part of Theorem \ref{T4.4}. Without loss of generality assume that $E^{*}(1, 1)$ is   the third positive equilibrium, which  implies the value $1$ is the maximal root of  Eq. \eqref{dai2.8} in the interval $(0, K)$.
If $E^{*}(1, 1)$ is center or focus type, then  $(K,b)\in \mathcal{D}$, where
\begin{equation}\label{cond4.6}
\mathcal{D}=\left\{(K,b)\left|1<K<3, \,\, \textup{max}\bigg\{0, 1-\frac{2}{K}\bigg\}<b<\textup{min}\bigg\{\frac{(K-1)^2}{4K},\ 2 K-3\bigg\}\right.\right\}.
\end{equation}
The  set $\mathcal{D}$  is an open region, that is, the shaded region surrounded by the curvilinear rectangle $ABCD$  in Fig. 4.3, where  $AB$ and $AD$ are two line segments, $BC$ and $CD$ are two curves defined by $b=1-2/K$ and $b=(K-1)^2/(4 K)$, respectively.
\begin{figure}\begin{center}
{\includegraphics[height=5cm,width=8.5cm]{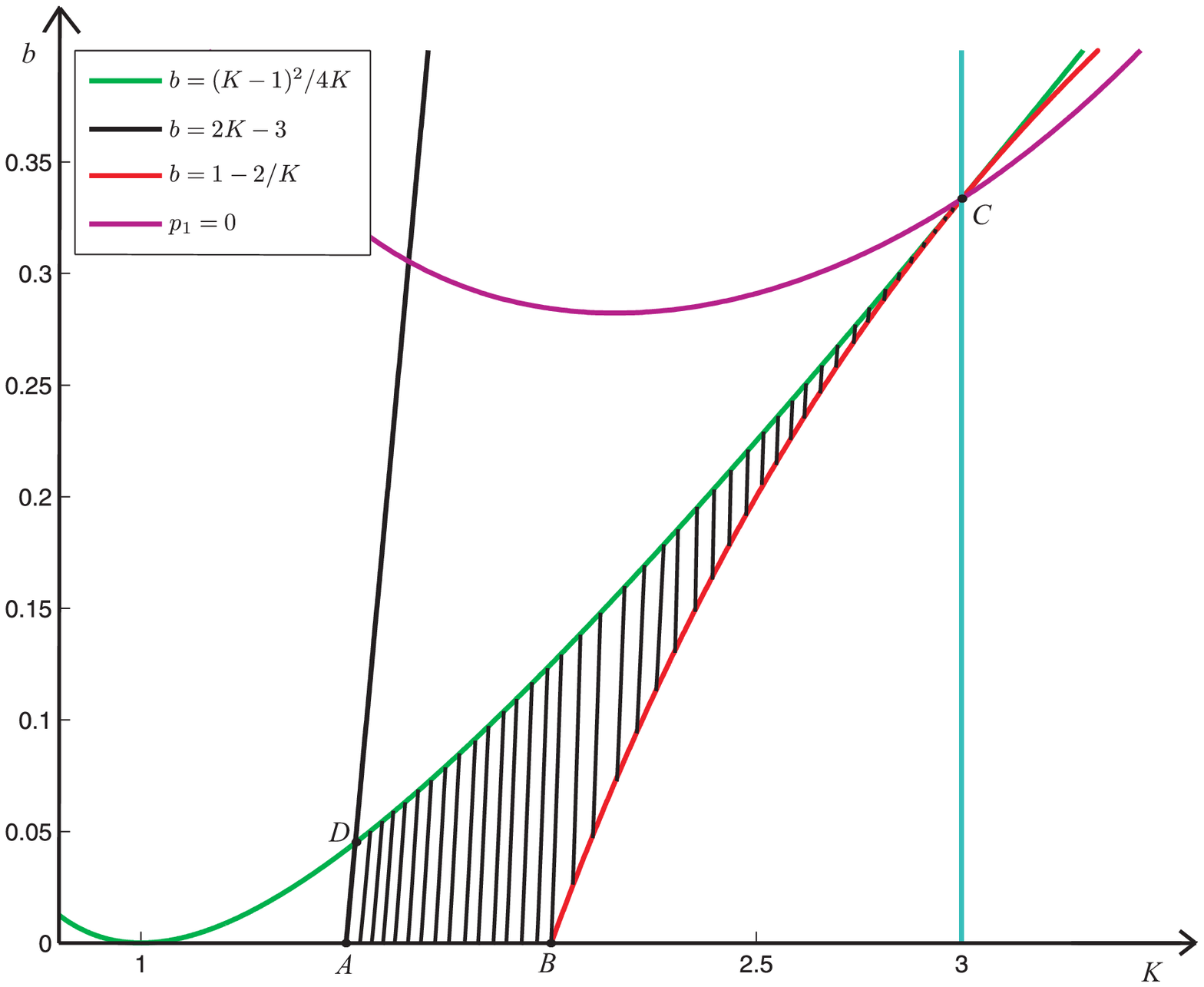}}
\end{center}
\begin{center}
{\small {\bf Fig. 4.3.} Schematic diagram of the region $\mathcal{D}$. }
\end{center}
\end{figure}

In what follows we will prove  $\upsilon_1(K, b)>0$  for $(K,b)\in\mathcal{D}$. We first study the absolute maximum and minimum values of the continuous function $\upsilon_1(K, b)$ on the closure of $\mathcal{D}$, denoted by $\overline{\mathcal{D}}$.
Since $\upsilon_1(K, b)$ is differentiable,  $\upsilon_1(K, b)$ has a maximum point and a minimum point at either the point  $(K,b)\in \mathcal{D}$ satisfying  $\frac{\partial \upsilon_1}{\partial K}=\frac{\partial \upsilon_1}{\partial b}=0$ or the  points on the boundary. Let
\begin{eqnarray}\label{dai4.3}
\begin{array}{ll}
\dfrac{\partial \upsilon_1}{\partial K}=b(12K{b}^{2}-{b}^{3}-8K b-15{b}^{2}+12K-3b-21)=0, \\[6pt]
\dfrac{\partial \upsilon_1}{\partial b}=
18{K}^{2}{b}^{2}-4K{b}^{3}-8{K}^{2}b-45K{b}^{2}+6{K}^{2}-6
K b+3{b}^{2}-21K-6b+15=0.
\end{array}
\end{eqnarray}
Denote by $S_3$ a semi-algebraic system whose polynomial equations, non-negative polynomial inequalities, positive polynomial inequalities and polynomial inequations are given by $F:=\Big[\frac{\partial \upsilon_1}{\partial K}, \frac{\partial \upsilon_1}{\partial b}\Big]$, $N:=[\ ]$, $P:=[K-1, 3-K, b, K b-K+2, 2K-b-3, (K-1)^2-4 K b]$, and $H:=[\ ]$, respectively, where $[\ ]$ represents the null set.
Using the \textit{RealRootIsolate} program in Maple with accuracy $1/10^{20}$,
we conclude that $\{\frac{\partial \upsilon_1}{\partial K}, \frac{\partial \upsilon_1}{\partial b}\}$ has no common real root in $\mathcal{D}$, implying that $\upsilon_1(K, b)$ has no extreme value in the interior of the curvilinear rectangle ABCD.
Therefore, the absolute maximum and minimum values of $\upsilon_1(K, b)$ must be on the boundary.
Noting that $3/2\leq K\leq3$, we have
\begin{eqnarray}\label{max4.11}
\begin{array}{ll}
\upsilon_1\big|_{\overline{AB}}=\upsilon_1\big|_{b=0}=3>0, \\[6pt]
\upsilon_1\big|_{\wideparen{BC}}=\upsilon_1\big|_{b=1-2/K}=\dfrac{8(K-1)^2(K-3)[(K-1)(K-3)-2]}{K^3}\geq 0,\\[6pt]
\upsilon_1\big|_{\wideparen{CD}}=\upsilon_1\big|_{b=
(K-1)^2/(4K)}=\dfrac{(K-3)\Psi_1(K)}{256K^3},\\[6pt]
\upsilon_1\big|_{\overline{AD}}=\upsilon_1\big|_{b=2K-3}=32(K-1)^2\Psi_2(K),
\end{array}
\end{eqnarray}
where $\Psi_1=23\,{K}^{7}-191\,{K}^{6}+715\,{K}^{5}-1555\,{K}^{4}+1077\,{K}^{3}-493
\,{K}^{2}+41\,K-1$, $\Psi_2=K^3-6K^2+9K-3$.
By Sturm's Theorem, we have $\Psi_1<0$ for  $K\in [3/2, 3]$, implying that $\upsilon_1\big|_{\wideparen{CD}}\geq 0$. Note that $A(3/2, 0)$ and $D((5+4\sqrt{2})/7, (8\sqrt{2}-11)/7)$.
For the points on the line segment $\overline{AD}$, we have $3/2\leq K\leq (5+4\sqrt{2})/7<11/7$. By Sturm's Theorem, we conclude that $\Psi_2>0$ for  $K\in [3/2, 11/7]$, which follows $\upsilon_1\big|_{\overline{AD}}>0$.
From the discussions above, we know $\upsilon_1(K,b)\geq 0$ for $(K,b)\in\overline{\mathcal{D}}$ and the minimum value $0$ of $\upsilon_1(K, b)$ occurs only at the boundary point $C(3, 1/3)$. Thus, we have $\upsilon_1(K, b)>0$ for $(K,b)\in \mathcal{D}$.

It follows from \eqref{dai2.13} that  $V_3(K, b)>0$ for $(K,b)\in \mathcal{D}$.  This implies $E^{*}(1, 1)$ cannot be a center and is an unstable multiple focus with multiplicity one, provided that it is center or focus type. Furthermore, we can perturb $s$ such that $V_1<0$  and hence one limit cycle bifurcates (see Fig. 4.4). This completes the proof.
\end{proof}
\begin{figure}
\begin{center}
{\includegraphics[height=7cm,width=8.5cm]{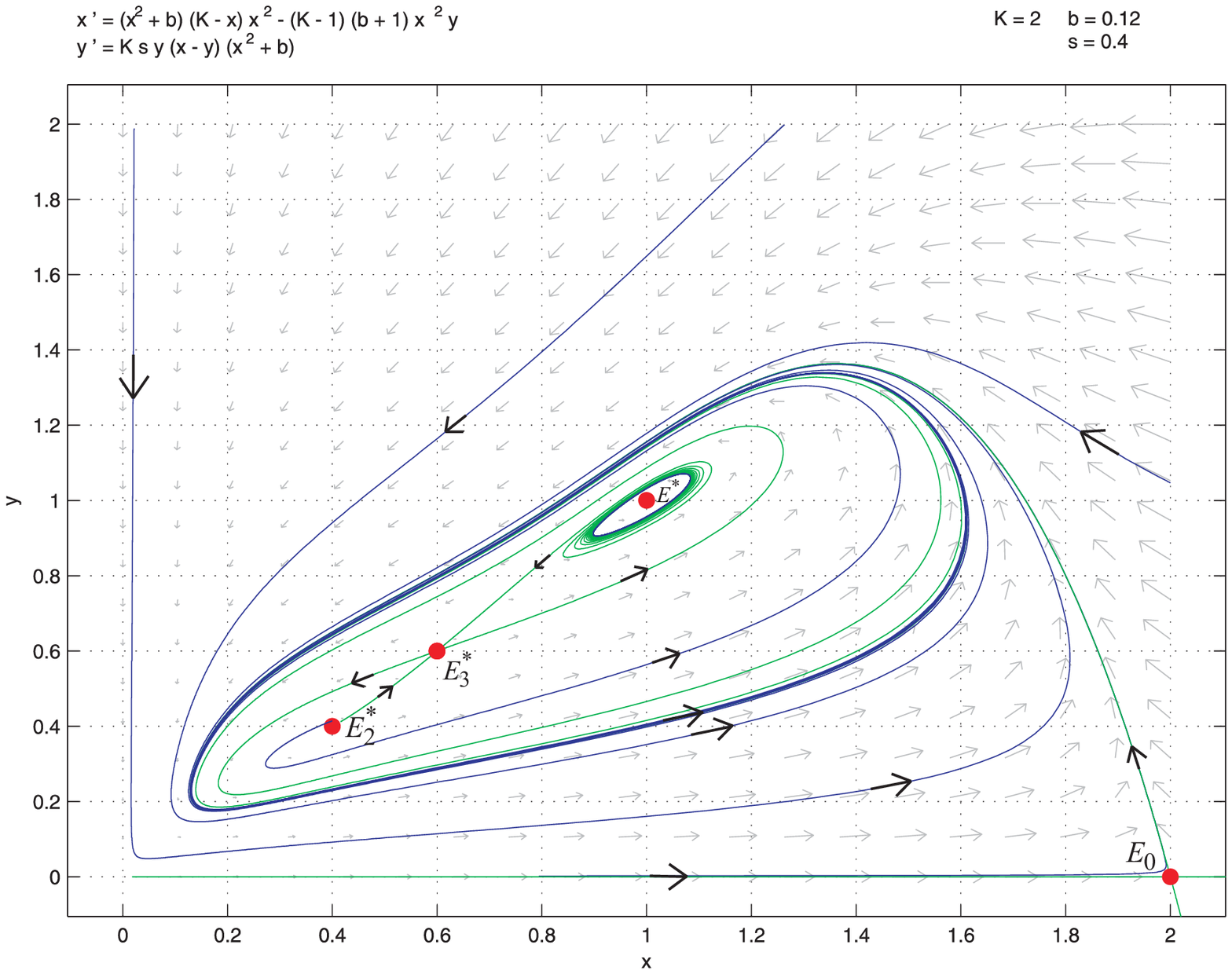}}
\end{center}
{\small {\bf Fig. 4.4.} Phase portraits for system \eqref{dai2.7} with $(K, b, s)=(2, 0.12, 0.4)$, which satisfy the conditions \eqref{cond4.6}: an unstable limit cycle enclosing the first positive equilibrium $E^{*}(1, 1)$ to the left of $E_0$, and a big stable limit cycle enclosing this limit cycle and three hyperbolic positive equilibria, bistability states.}
\end{figure}
\begin{rema}\label{rem3}
When the model has three distinct positive equilibria, the phenomenon that one limit cycle can bifurcate from the first positive equilibrium has also been observed by Collings \cite{Collings1997}, Li and Xiao \cite{Li2007} and Huang et al. \cite{Huang2016}. The phenomenon that one limit cycle can bifurcate from the third positive equilibrium has also been observed by Li and Xiao \cite{Li2007} and Huang et al. \cite{Huang2016}.
However, the phenomenon that two limit cycles can bifurcate from the first positive equilibrium, observed in the present paper, has not been found by other authors.
\end{rema}

\subsubsection{Hopf bifurcation at the two center or focus type equilibria simultaneously}
 If system \eqref{dai2.7} has three distinct positive equilibria, then two of them are anti-saddle. In this subsection, we study system \eqref{dai2.7}  when both   anti-saddle positive equilibria are center or focus type.  We will study the number of limit cycles bifurcating from each of them and discuss all possible distributions of limit cycles. Our results are as follows.
\begin{theo}\label{T4.5}
 If both  anti-saddle equilibria of \eqref{dai2.7}  are center or focus type, then
\begin{itemize}
\item[(1)] they are unstable multiple foci with multiplicity one;
\item[(2)] one limit cycle can bifurcate from each of them simultaneously.
\end{itemize}
\end{theo}
\begin{proof} To simplify the computations, we introduce the new parameters as follows.  Under  the assumption that system \eqref{dai2.7} has three distinct positive equilibria,  we assume that  Eq. \eqref{dai2.8} has three positive zeros $1,\,\alpha$ and $\beta$ in the interval $(0, K)$  with $1<\alpha<\beta<K$. Then $E^{*}_2(\alpha, \alpha)$ is a saddle,   $E^{*}(1, 1)$ and $E^{*}_3(\beta, \beta)$ are both anti-saddle, respectively.
From Eq. \eqref{dai2.8} and Vieta's formulas for quadratic polynomial, we have
\begin{equation*}
\alpha\beta=K b, \quad
\alpha+\beta=K-1,
\end{equation*}
i.e.,
\begin{equation}\label{dai4.5}
K=\alpha+\beta+1, \quad
b=\dfrac{\alpha\beta}{\alpha+\beta+1}.
\end{equation}
By \eqref{cond},  system \eqref{dai2.7} has three distinct positive equilibria if and only if \eqref{dai4.5} holds, and
\begin{equation}\label{cond4.7}
1<\alpha<\beta<K,\quad  s>0.
\end{equation}
Using \eqref{dai4.5} and the  time scaling transformation
$dt=d\tau/(\alpha+\beta+1)$, system \eqref{dai2.7} is  reduced to the following equivalent polynomial  system (we will still use $t$ to denote $\tau$ for ease of notation)
\begin{eqnarray}\label{dai4.6}
\left\{
\begin{array}{ll}
\dot{x}=\left((\alpha+\beta+1)x^2+\alpha\beta\right) \left( (\alpha+\beta+1)-x\right) {x}^{2}-(\alpha+\beta )(\alpha\beta+\alpha+\beta+1) {x}^{2}y, \\[6pt]
\dot{y}= ( \alpha+\beta+1)s y (x-y) \left(( \alpha+\beta+1)x^2+\alpha\beta\right),
\end{array}
\right.
\end{eqnarray}
where the parameters $K$, $\alpha$, $\beta$ and $s$ satisfy \eqref{cond4.7}.
Next we study system \eqref{dai4.6}, instead of system \eqref{dai2.7}.

For convenience, we denote $V_1^{(1)}\triangleq\textrm{Tr}(J(E^{*}))$ and $V_1^{(3)}\triangleq\textrm{Tr}(J(E^{*}_3))$, where $J(E^{*})$ is the Jacobian matrix of system \eqref{dai4.6} at $E^{*}$.
 From system \eqref{dai4.6}, it's not difficult to obtain
\begin{eqnarray}\label{focus3}
\left\{
\begin{array}{ll}
V_1^{(1)}= 2\,{\alpha}^{2}+3\,\alpha\,\beta+2\,{\beta}^{2}+\alpha+\beta-1-
 \left( \alpha+1 \right) \left( \beta+1 \right) \left( \alpha+\beta+
1 \right) s,\\[6pt]
V_1^{(3)}= {\beta}^{3} \left( 2\,{\alpha}^{2}+\alpha\,\beta-{\beta}^{2}+3\,\alpha
+\beta+2 \right) -{\beta}^{2} \left( \beta+1 \right)  \left( \alpha+
\beta \right)  \left( \alpha+\beta+1 \right) s.
\end{array}
\right.
\end{eqnarray}
If both   $E^{*}(1, 1)$ and $E_3^{*}(\beta, \beta)$ are center or focus type, then  $V_1^{(1)}=V_1^{(3)}=0$, which implies
\begin{eqnarray}\label{dai4.7}
\left\{
\begin{array}{ll}
s={\dfrac{2\left(17\,{\alpha}^{3}
+19\,{\alpha}^{2}+9\,\alpha-9+\left(5\alpha-3\right)\sqrt {w}\right)}{\left( {\alpha}^{2}-2\,\alpha+3+\sqrt {w} \right)  \left( 3\,{\alpha}^{2}+4\,\alpha+3+
\sqrt {w} \right) }}\triangleq s_0, \\[16pt]
\beta={\dfrac {{\alpha}^{2}-4\,\alpha-3+\sqrt {w}}{2(\alpha+3)}}
\triangleq \beta_0,
\end{array}
\right.
\end{eqnarray}
where $w=9\,{\alpha}^{4}+20\,{\alpha}^{3}+18\,{\alpha}^{2}+12\,\alpha+9>0$.
It's obvious that the denominators of the right-hand side of Eq. \eqref{dai4.7} are positive.

 We first prove the claim (1). It suffices to prove $V_{3}^{(1)}>0$ and $V_{3}^{(3)}>0$, where $V_{3}^{(1)}$ and $V_{3}^{(3)}$ are the first  Lyapunov constants of system \eqref{dai4.6} at  $E^{*}(1, 1)$ and $E_3^{*}(\beta, \beta)$, respectively.

Note that the expressions \eqref{dai4.7} are  complicated  for the computation of  the Lyapunov constants. In order to simplify  calculation, we are in no hurry to substitute \eqref{dai4.7} to the system \eqref{dai4.6}
to calculate  $V_{3}^{(1)}$ and $V_{3}^{(3)}$. We will calculate them separately.

Under the condition $V_1^{(1)}=0$, we can get
\begin{equation}\label{dai4.8}
V_{3}^{(1)}=\dfrac {(\alpha+\beta)^2F_1}{4(\alpha-1)(\beta-1)},
\end{equation}
where we have eliminated the parameter $s$, and $F_1$ is a polynomial in $\alpha$ and $\beta$ of degree $8$ and given in the Appendix B. Similarly, with the condition $V_{1}^{(3)}=0$, we can obtain
\begin{equation}\label{dai4.9}
V_{3}^{(3)}=\dfrac {(\alpha+1)^2{\beta}^{5} G_1}{4(\beta-1)(\beta-\alpha)},
\end{equation}
where $G_1$ is a polynomial in $\alpha$ and $\beta$ of degree $8$ and given in the Appendix B.
From  \eqref{dai4.7}, we know that $\beta$ is the unique positive solution of the following quadratic polynomial equation
\begin{equation}\label{dai4.10}
\Phi:=\left(\alpha+3 \right) {\beta}^{2}- \left( {\alpha}^{2}-4\,\alpha-3
 \right) \beta-2\,{\alpha}^{3}-{\alpha}^{2}+\alpha=0.
\end{equation}
Since $1<\alpha<\beta<K$, the signs of $V_1^{(1)}$ and $V_1^{(3)}$ are the same as that of $F_1$ and $G_1$, respectively. Hence we only need to prove $F_1>0$ and $G_1>0$ under the conditions \eqref{cond4.7} and \eqref{dai4.10}.

We first prove $F_1>0$.
Applying  pseudo division and by command $prem(F_1, \Phi, \beta)$ in Maple, we obtain
\begin{equation}\label{dai4.11}
(\alpha+3)^4F_1=q_1\Phi+(\alpha+1)^3(h_1\beta+h_2),
\end{equation}
where $h_1$ and $h_2$ are polynomials in $\alpha$ of degree $8$ and $9$, respectively, and $q_1$ is a polynomial in $\alpha, \beta$ of degree $9$.
It follows from \eqref{dai4.10} and \eqref{dai4.11} that $F_1=(\alpha+1)^3(h_1\beta+h_2)/(\alpha+3)^4$.
By Sturm's Theorem, we know that $h_1>0$ for $\alpha\in(1, +\infty)$. Noting that $\beta>\alpha$, we have $h_1\,\beta+h_2>h_1\,\alpha+h_2\triangleq h_3$, where $h_3=277\alpha^9-1043 \alpha^8+1296 \alpha^7-262 \alpha^6-858 \alpha^5+96 \alpha^4+576\alpha^3+486\alpha^2+405\alpha+243$.
Similarly, by Sturm's Theorem, we can conclude that $h_3(\alpha)>0$ for $\alpha\in(1, +\infty)$, which yields $F_1>0$.

Next we prove $G_1>0$. Let $\varepsilon=(\beta-\alpha)/(\alpha-1)$, i.e., $\alpha=(\varepsilon+\beta)/(\varepsilon+1)$. Substituting it to $G_1$, we obtain  $G_1=(\beta-1)\overline{G}_1/(\varepsilon+1)^{5}$, where $\overline{G}_1$ is collected by $\varepsilon$ as below.
$$
\begin{aligned}
\overline{G}_1=&(3{\beta}^{7}+21{\beta}^{6}+57{\beta}^{5}+63{\beta}^{4}
+5{\beta}^{3}-13{\beta}^{2}+15\beta-23 ){\varepsilon}^{5}+(24{\beta}^{7}+141{\beta}^{6}+294{\beta}^{5}+199{\beta}^{4}\\
&-20{\beta}^{3}+19{\beta}^{2}-58\beta-23){\varepsilon}^{4}+(75{\beta}^{7}+354{\beta}^{6}+
527{\beta}^{5}+206{\beta}^{4}+17{\beta}^{3}-90{\beta}^{2}-59\beta-\\
&6){\varepsilon}^{3}+(114{\beta}^{7}+408{\beta}^{6}+414{\beta}^{5}+136{\beta}^{4}-82{
\beta}^{3}-80{\beta}^{2}-14\beta){\varepsilon}^{2}+(84{\beta}^{7}+216{\beta}^{6}+
168{\beta}^{5}\\
&-8{\beta}^{4}-60{\beta}^{3}-16{\beta}^{2})\varepsilon+
24{\beta}^{7}+48{\beta}^{6}+16{\beta}^{5}-16{\beta}^{4}-8{\beta}^{3}.
\end{aligned}
$$
Obviously, all the coefficients of the power of $\varepsilon$ in $\overline{G}_1$ are positive since $\beta>1$. This implies $\overline{G}_1>0$, which yields $G_1>0$.
This  completes the proof of the claim (1).

Next we are going to  prove the claim (2). For any given values of $K$ and $\alpha$, let $(s, \beta)=(s_0+\varepsilon, \beta_0)$, where $s_0$ and $\beta_0$ are defined by \eqref{dai4.7}, and $\varepsilon$ is a perturbation parameter. From \eqref{focus3}, we have
\begin{equation}\label{dai4.20}
V_1^{(1)}=-(\alpha+1) (\beta_0+1) (\alpha+\beta_0+1), \quad
V_1^{(3)}=-\beta_0^{2} (\beta_0+1) (\alpha+\beta_0) (\alpha+\beta_0+1).
\end{equation}
By continuity, we can take $\varepsilon$ small enough such that $V_{3}^{(1)}>0$ and $V_{3}^{(3)}>0$.
From \eqref{dai4.20},  it follows from $K>\beta_0>\alpha>1$ that $V_{1}^{(1)}<0$ and $V_{1}^{(3)}<0$ hold as long as $\varepsilon>0$.
That is to say, for any given values of $K$ and $\alpha$, there exist parameter values $(s, \beta)$ near $(s_0, \beta_0)$ such that $V_{1}^{(1)}V_{3}^{(1)}<0$ and $V_{1}^{(3)}V_{3}^{(3)}<0$. This implies one limit cycle can bifurcate from $E^{*}(1, 1)$ and $E_3^{*}(\beta, \beta)$ simultaneously.  This completes the proof.
\end{proof}
\begin{figure}
\begin{minipage}[t]{0.5\linewidth}
\centering
\includegraphics[width=2.8in]{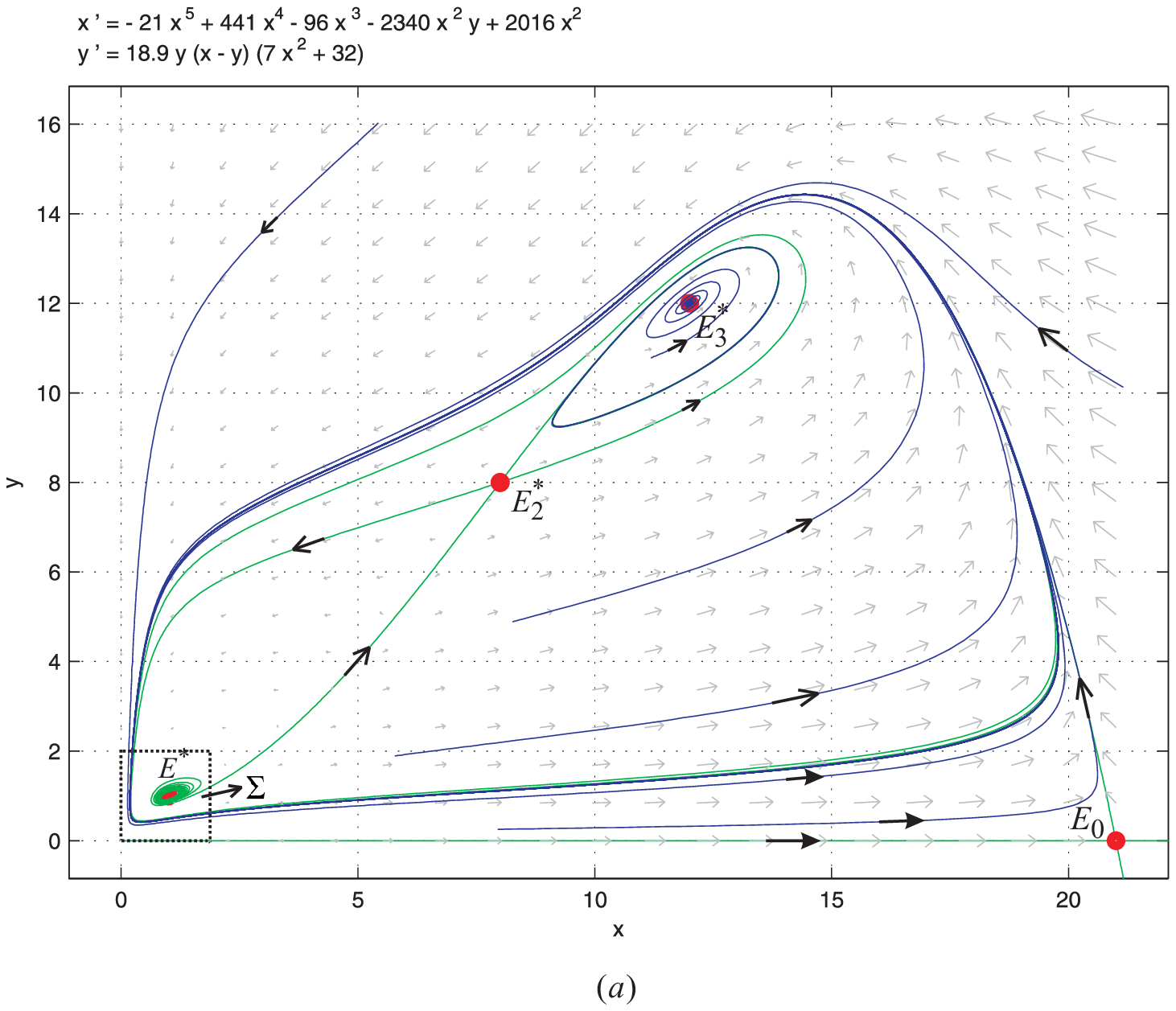}
\end{minipage}%
\begin{minipage}[t]{0.5\linewidth}
\centering
\includegraphics[width=2.8in]{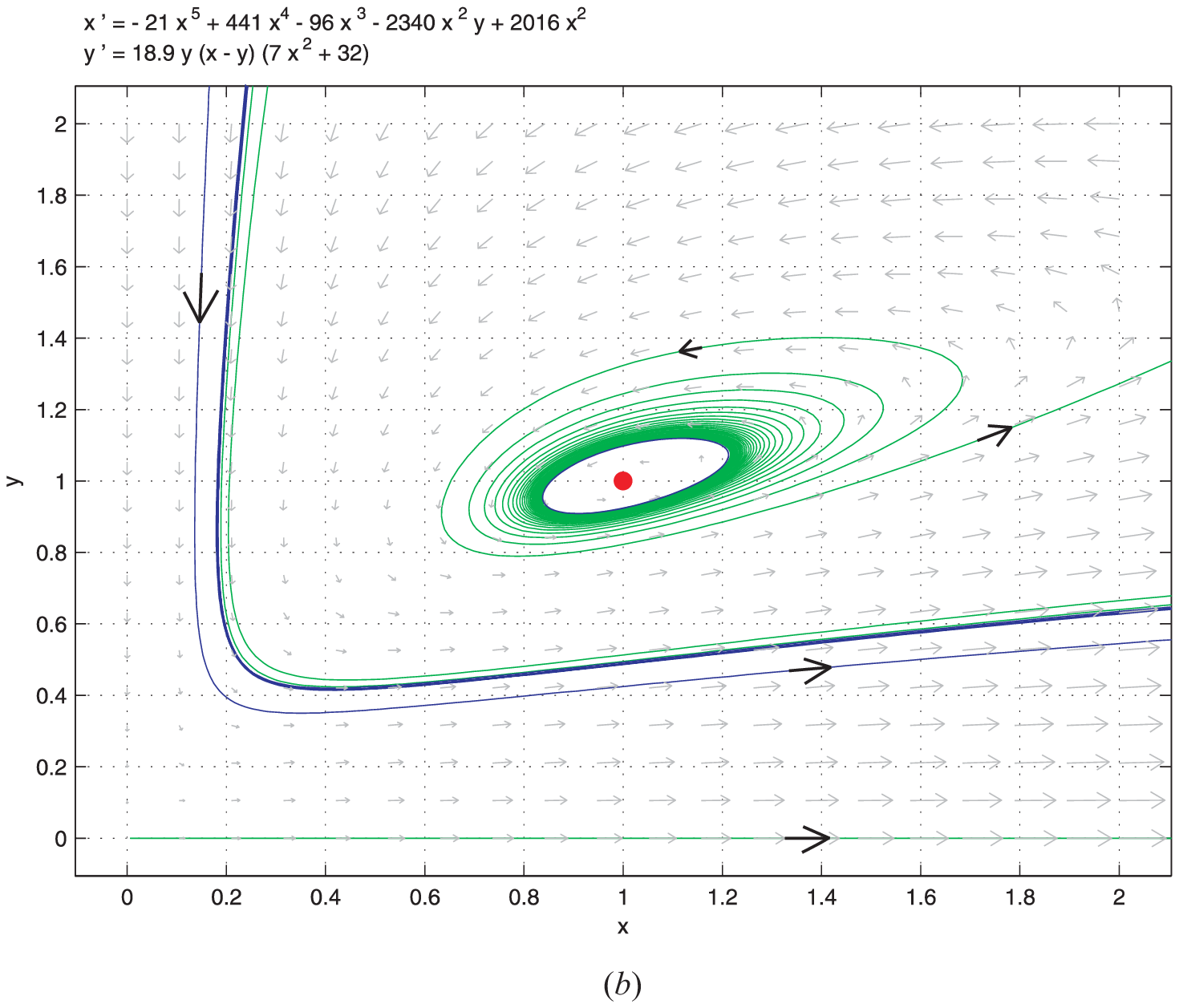}
\end{minipage}
{\small {\bf Fig. 4.5.}
Phase portraits for system \eqref{dai4.6} with $\alpha=8$: (a) choose $(s, \beta)=(0.3, 12)$ near $(s_0, \beta_0)\approx(0.299318, 12.314152)$, which satisfy the conditions \eqref{cond4.7},  one big stable limit cycle enclosing two unstable limit cycles; (b) the magnification of the region $\Sigma$ in the left figure.}
\end{figure}
\begin{rema}\label{rem4}
If system \eqref{dai2.7} has three distinct positive equilibria, then it follows from
Theorem \ref{T4.5} that one limit cycle can bifurcate from each of the two anti-saddle equilibria simultaneously (see Fig. 4.5). This  has not been observed  by Collings \cite{Collings1997}, Li and Xiao \cite{Li2007} and Huang et al. \cite{Huang2016}. Numerical simulations show that there is also a big stable limit cycle enclosing them (see Fig. 4.5), which is a new phenomenon observed in the present paper.  If we can rule out the possibility  that there are  two homoclinic loop connecting with the saddle $E_2^*$, called eight-loop in general, then Poincar\'{e}-Bendixson Theorem implies the existence of the big stable limit cycle.
 Unfortunately, we can't rigorously prove it.
\end{rema}

\section{Discussion}

It is natural that the predator-prey interaction has the tendency or potential to
induce periodic oscillations, which has been confirmed by experimental data.
The existence of limit cycles in predator-prey systems can be used to explain this phenomenon and has been extensively studied, see \cite{Albrecht1973,May1972,Hsu1998}, etc.

The models with Holling types I, II and III functional responses showed qualitatively similar behaviors, which suggests that the model framework provided by system \eqref{dai1.1} is rather robust with respect to these three functional response types. The  different behavior of the system with Holling type IV functional response at higher levels of prey interference is principally because this type functional response is nonmonotonic, while the other three type functional responses are monotonically increasing and tend to a maximum.
The simplified  Holling type IV functional response has the same monotonicity as the Holling type IV functional response and is simpler since it contains less parameters.
Correspondingly, the system \eqref{dai1.1} with simplified  Holling type IV functional response
should have different behavior  from the models with Holling types I, II and III functional responses. This has been confirmed by Li and Xiao \cite{Li2007} and Huang et al. \cite{Huang2016} through theoretical analysis and numerical simulations.

In this paper, we studied the Hopf bifurcation and global dynamics of system \eqref{dai1.2}  with no degenerate positive equilibrium.
It is shown that Hopf bifurcation can occur at each weak focus. Hence, more complicated and new dynamics, such as the coexistence of three hyperbolic positive equilibria and two limit cycles enclosing only one of them, or a big stable limit cycle enclosing two unstable limit cycles, have  been observed in the present paper.
And more, we also rigorously proved the  existence of some phenomena that have been observed ever before, see Remarks \ref{rem2} and \ref{rem3}.
These phenomena not only support the numerical observations of Collings \cite{Collings1997} that there are different kinds of population oscillations and outbreaks in response to increasing temperature-dependent parameters and population perturbations (initial population density), but also complete the bifurcation analysis of Li and Xiao \cite{Li2007} and Huang et al. \cite{Huang2016} on the model.
Therefore, our results can be a complement to the works by Collings \cite{Collings1997}, Li and Xiao \cite{Li2007} and Huang et al. \cite{Huang2016} on the model and show by numerical simulation  that the maximum number of limit cycles in this class of predator-prey system may be at least $3$, which improve the preceding results  that this number is at least $2$.

There are two interesting problems that remain open. One is to prove the
the existence of the big stable limit cycle for some given parameter values, see Remark \ref{rem4}.
We conjecture that this problem may be solved by constructing an appropriate inner boundary of an annular region. The other is the problem of the global stability of $E^{*}(1, 1)$ if \eqref{cond4.12} holds. We leave them for future work.

\section*{Acknowledgements}

The authors thank to Dr. Xiuli  Cen for her helpful discussions and comments.

\section*{Appendix A}

The expressions of $\upsilon_1$,  $\upsilon_2$ in Eq. \eqref{dai2.13}  are listed as follows.
\begin{eqnarray*}
\begin{aligned}
\upsilon_1=&6\,{K}^{2}{b}^{3}-K{b}^{4}-4\,{K}^{2}{b}^{2}-15\,K{b}^{3}+6\,{K}^{2}b-
3\,K{b}^{2}+{b}^{3}-21\,Kb-3\,{b}^{2}+15\,b+3,\\
\upsilon_2=&220\,{K}^{6}{b}^{7}-152\,{K}^{5}{b}^{8}+21\,{K}^{4}{b}^{9}+2622\,{K}^{
6}{b}^{6}-3413\,{K}^{5}{b}^{7}+1184\,{K}^{4}{b}^{8}-91\,{K}^{3}{b}^{9}
-5534\,{K}^{6}{b}^{5}-\\
&12074\,{K}^{5}{b}^{6}+10498\,{K}^{4}{b}^{7}-1845
\,{K}^{3}{b}^{8}+4\,{K}^{2}{b}^{9}+6428\,{K}^{6}{b}^{4}+29321\,{K}^{5}
{b}^{5}+16227\,{K}^{4}{b}^{6}-\\
&8248\,{K}^{3}{b}^{7}-13\,{K}^{2}{b}^{8}-
4592\,{K}^{6}{b}^{3}-37012\,{K}^{5}{b}^{4}-67760\,{K}^{4}{b}^{5}+6762
\,{K}^{3}{b}^{6}-3010\,{K}^{2}{b}^{7}+4\,K{b}^{8}\\
&+1094\,{K}^{6}{b}^{2}
+35821\,{K}^{5}{b}^{3}+80087\,{K}^{4}{b}^{4}+88726\,{K}^{3}{b}^{5}-
22986\,{K}^{2}{b}^{6}+289\,K{b}^{7}+82\,{K}^{6}b-\\
&10058\,{K}^{5}{b}^{2}
-113282\,{K}^{4}{b}^{3}-80428\,{K}^{3}{b}^{4}-54164\,{K}^{2}{b}^{5}+
1108\,K{b}^{6}-8\,{b}^{7}-1217\,{K}^{5}b+\\
&38497\,{K}^{4}{b}^{2}+182632
\,{K}^{3}{b}^{3}+43778\,{K}^{2}{b}^{4}-1209\,K{b}^{5}+13\,{b}^{6}+7067
\,{K}^{4}b-80562\,{K}^{3}{b}^{2}-\\
&151882\,{K}^{2}{b}^{3}-20614\,K{b}^{4
}+48\,{b}^{5}+69\,{K}^{4}-20395\,{K}^{3}b+100022\,{K}^{2}{b}^{2}+53967
\,K{b}^{3}+525\,{b}^{4}-\\
&663\,{K}^{3}+32316\,{K}^{2}b-72528\,K{b}^{2}-
4416\,{b}^{3}+1983\,{K}^{2}-27735\,K\,b+24423\,{b}^{2}-2466\,K+\\
&10584\,b+1215.
\end{aligned}
\end{eqnarray*}

\section*{Appendix B}

The expressions of $F_1$ in Eq. \eqref{dai4.8}, $G_1$ in Eq. \eqref{dai4.9} and $h_1$, $h_2$, $q_1$  in Eq. \eqref{dai4.11} are listed as follows.
\begin{eqnarray*}
\begin{aligned}
F_1=&6\,{\alpha}^{5}{\beta}^{3}+11\,{\alpha
}^{4}{\beta}^{4}+6\,{\alpha}^{3}{\beta}^{5}-4\,{\alpha}^{5}{\beta}^{2}-15\,{\alpha}^{4}{\beta
}^{3}-15\,{\alpha}^{3}{\beta}^{4}-4\,{\alpha}^{
2}{\beta}^{5}+6\,{\alpha}^{5}\,\beta+9\,{\alpha}^{4}{\beta}^{2}-
2\,{
\alpha}^{3}{\beta}^{3}+\\
&9\,{\alpha}^{2}{\beta}^{4}+6\,\alpha{\beta}^{5}+3\,
{\alpha}^{4}\,\beta-12\,{\alpha}^{3}{\beta}^{2}-12\,{\alpha}^{2}{\beta}^{3}+3
\,\alpha\,{\beta}^{4}-12\,{\alpha}^{3}\beta-34\,{\alpha}^{2}{\beta}^{2
}-12\,\alpha\,{\beta}^{3}+3\,{\alpha}^{3}+3\,{\beta}^{3}+\\
&9\,{\alpha}^{
2}+18\,\alpha\,\beta+9\,{\beta}^{2}+9\,\alpha+9\,\beta+3,\\
G_1=&3\,{\alpha}^{3}{\beta}^{5}+9\,{\alpha}^{2}{\beta}^{6}+9\,\alpha\,{
\beta}^{7}+3\,{\beta}^{8}+6\,{\alpha}^{5}{\beta}^{2}+3\,{\alpha}^{4}{
\beta}^{3}-12\,{\alpha}^{3}{\beta}^{4}+18\,\alpha\,{\beta}^{6}+9\,{
\beta}^{7}-4\,{\alpha}^{5}\beta+9\,{\alpha}^{4}{\beta}^{2}-\\
&12\,{\alpha
}^{3}{\beta}^{3}-34\,{\alpha}^{2}{\beta}^{4}+9\,{\beta}^{6}+6\,{\alpha
}^{5}-15\,{\alpha}^{4}\beta-2\,{\alpha}^{3}{\beta}^{2}-12\,{\alpha}^{2
}{\beta}^{3}-12\,\alpha\,{\beta}^{4}+3\,{\beta}^{5}+11\,{\alpha}^{4}-
15\,{\alpha}^{3}\beta+\\
&9\,{\alpha}^{2}{\beta}^{2}+3\,\alpha\,{\beta}^{3
}+6\,{\alpha}^{3}-4\,{\alpha}^{2}\beta+6\,\alpha\,{\beta}^{2},\\
h_1=&139 \,\alpha^8-562 \,\alpha^7+854 \,\alpha^6-372 \,\alpha^5-306 \,\alpha^4-18 \,\alpha^3+486 \,\alpha^2+243,\\
h_2=&138 \,\alpha^9-481 \,\alpha^8+442\, \alpha^7+110 \,\alpha^6-552\, \alpha^5+114 \, \alpha^4+90 \,\alpha^3+486 \,\alpha^2+162 \,\alpha+243.
\end{aligned}
\end{eqnarray*}


\begin{thebibliography}{{}}
\bibitem {Albrecht1973} F. Albrecht, H. Gatzke, and N. Wax, Stable limit cycles in prey-predator populations, Science 181 (1973) 1073-1074.

\bibitem {Andrews1968}  J. F. Andrews, A mathematical model for the continuous culture of microorganisms utilizing inhibitory substrates, Biotechnol. Bioeng., 10 (1968), pp. 707-723.

\bibitem {Collins1971} G. Collins, The Calculation of Multivariate Polynomial Resultants, Journal of the Acm, 18 (1971) 515-532.

\bibitem {Collings1997} J. B. Collings, The effects of the functional response on the bifurcation behavior of a mite predator-prey interaction model, J. Math. Biol. 36 (1997) 149-168.

\bibitem {Dai2017}  Y. Dai, Y. Zhao, B. Sang,  Global dynamics for a predator-prey system of Leslie type with generalized Holling type III functional response, preprint.

\bibitem {GelfandIM1994} I. Gelfand, M. Kapranov, A. Zelevinsky, Discriminants, Resultants and Multidimensional Determinants, Birkh$\ddot{a}$ser, Boston, 1994.

\bibitem {Hoyt1967} S. C. Hoyt, Population studies of five mite species on apple in Washington. In: Proc second int cong acarology. Sutton Bonington: England, 1967, p. 117-33.

\bibitem {Hoyt1969} S. C. Hoyt, Integrated Chemical Control of Insects and Biological Control of Mites on Apple in Washington. Journal of Economic Entomology, 1969, 62 (1) 74-86.

\bibitem {Hsu1995} S. Hsu, T. Huang, Global stability for a class of predator-prey system, SIAM J. Appl. Math. 55 (1995) 763-783.

\bibitem {Hsu1998} S.B. Hsu, T.W. Hwang, Uniqueness of limit cycles for a predator-prey system of Holling and Lesile type, Can. Appl. Math. Q. 6 (1998) 91-117.

\bibitem {Hsu1999} S.B. Hsu, T.W. Hwang, Hopf bifurcation analysis for a predator-prey system of Holling and Leslie type. Taiwan J Math 3 (1999) 35-53.

\bibitem {Huang2016} J. Huang, X. Xia, X. Zhang, Bifurcation of codimension 3 in a predator-prey system of leslie type with simplified Holling type IV functional response, Int. J. Bifurcat. Chaos, 26 (2016) 1650034.

\bibitem {LaSalle1961} LaSalle, J. P. and S. Lefschetz. Stability by Liapunov's Direct Method. Academic Press, New York, 1961.

\bibitem {Li2007}  Y. Li, D. Xiao, Bifurcations of a predator-prey system of Holling and Leslie types, Chaos Solitons Fractals 34 (2007) 606-620.

\bibitem {Holling1965} C.S. Holling, The functional response of predators to prey density and its role in mimicry and population regulation. Mem Ent Soc Can 46 (1965) 1-60.
    
\bibitem {Leslie1948} P.H. Leslie, Some further notes on the use of matrices in population mathematics. Biometrika 35 (1948) 213-245.

\bibitem {Leslie1960} P.H. Leslie, J.C. Gower, The properties of a stochastic model for the predator-prey type of interaction between two species. Biometrika 47 (1960) 219-34.

\bibitem {Liu2001} Y. Liu, Theory of center-focus in a class of high order singular points and infinity. Sci. in China, 31 (2001) 37-48.

\bibitem {Liu2008} Y. Liu, J. Li, W. Huang, Singular point values, center problem and bifurcations of limit cycles of two dimensional differential autonomous systems, Science Press, Beijing, 2008.

\bibitem {Lu2007} Z. Lu, B. He, Y. Luo, L. Pan,  An algorithm of real root isolation for polynomial systems with applications to the construction of limit cycles, Symbolic-Numeric Computation 232 (2007), 131-147.

\bibitem {May1972} R. May, Limit cycles in predator-prey communities, Science 177 (1972) 900-902.

\bibitem {May1973} R. May, Stability and Complexity in Model Ecosystems, Princeton University Press, Princeton, NJ, 1973.

\bibitem {Ruan2001} S. Ruan, D. Xiao, Global analysis in a predator-prey system with nonmonotonic functional response, SIAM J. Appl. Math. 61 (2001) 1445-1472.

\bibitem {Sang2016} B. Sang, Q. Wang, The center-focus problem and bifurcation of limit cycles in a class of 7th-degree polinomial systems, J. Appl. Anal. Comput. 6 (2016) 817-826.

\bibitem {Sokol1980} W. Sokol, J. A. Howell, Kinetics of phenol oxidation by washed cells. Biotechnology and Bioengineering, 1980, 23 (9) 2039-2049.

\bibitem {Volterra1927} V. Volterra, Variazioni e fluttuazioni del numero dindividui in specie aniMali conviventi, Accademia Dei Lincei III 6 (1927) 31-113.

\bibitem {Wollkind1978} D. J. Wollkind, Logan J A. Temperature-dependent predator-prey mite ecosystem on apple tree foliage. Journal of Mathematical Biology, 1978, 6 (3) 265-283.

\bibitem {Wollkind1988} D. J. Wollkind, Metastability in a temperature-dependent model system for predator-prey mite outbreak interactions on fruit trees. Bulletin of Mathematical Biology, 1988, 50(4) 379-409.

\bibitem {Xiao2001} D. Xiao, S. Ruan, Global dynamics of a ratio-dependent predator-prey system, J. Math. Biol. 43 (2001) 268-290.

\bibitem {ZhangZF1992} Z. Zhang, T. Ding, W. Huang, Z. Dong, Qualitative Theory of Differential Equations, Transl. Math. Monogr., vol. 101, American Mathematical Society, Providence, RI, 1992.

\bibitem {Zhao2015} Y. Zhao, Z. Feng, Y. Zheng, X. Cen, Existence of limit cycles and homoclinic bifurcation inaplant-herbivore model with toxin-determined functional response, J. Differential Equations 258 (2015) 2847-2872.

\end{thebibliography}
\end{document}